\documentclass{amsart}
\pdfoutput=1
\usepackage[utf8]{inputenc}
\usepackage[T1]{fontenc}
\usepackage{amsthm}
\usepackage{thmtools}

\usepackage{stmaryrd}
\usepackage{float}
\usepackage{makecell}
\usepackage{caption}
\usepackage{mathrsfs}
\usepackage{mathtools}

\usepackage{todonotes}

\usepackage[margin=1.5in]{geometry}
\usepackage{tikz,tikz-3dplot,tikz-cd}
\usepackage{enumitem}

\usepackage[
backend=biber,
style=numeric,
sorting=ynt
]{biblatex}

\usepackage[%
	colorlinks,%
	bookmarksdepth=3,%
	citecolor=blue%
]{hyperref}

\usepackage[capitalise]{cleveref}

\theoremstyle{plain}
\newtheorem{theor}{Theorem}[section]
\newtheorem{conje}[theor]{Conjecture}
\newtheorem{proposition}[theor]{Proposition}

\newtheorem{lemma}[theor]{Lemma}

\theoremstyle{definition}

\newtheorem{example}[theor]{Example}
\newtheorem{remark}[theor]{Remark}
\newtheorem{definition}[theor]{Definition}

\crefname{defprop}{definition--proposition}{definition--propositions}
\Crefname{defprop}{Definition--Proposition}{Definition--Propositions}

\crefname{theor}{theorem}{theorem}
\Crefname{theor}{Theorem}{Theorem}

\newcommand{\V}{\mathcal{V}}

\newcommand{\QQ}{\mathbb{Q}}

\newcommand{\CC}{\mathbb{C}}
\newcommand{\OO}{\mathcal{O}}
\newcommand{\PP}{\mathbb{P}}
\newcommand{\ZZ}{\mathbb{Z}}
\newcommand{\RR}{\mathbb{R}}

\newcommand{\divv}{\operatorname{div}}
\newcommand{\rk}{\operatorname{rk}}
\newcommand{\Cl}{\operatorname{Cl}}
\newcommand{\D}{\mathcal{D}}
\newcommand{\Pic}{\operatorname{Pic}}
\newcommand{\Hom}{\operatorname{Hom}}
\newcommand{\cone}{\operatorname{cone}}
\newcommand{\conv}{\operatorname{conv}}

\newcommand{\Nef}{\operatorname{Nef}}
\newcommand{\Eff}{\operatorname{Eff}}

\title[The maximum Cox complexity of a Fano-type variety, the Mukai conjecture]{The maximum Cox complexity of a Fano-type variety and the Mukai conjecture}

\author{Heath Pearson}
\address{School of Mathematical Sciences\\University of Nottingham\\ Nottingham\\NG7 2RD\\UK\\}
\email{heath.pearson@nottingham.ac.uk}
\urladdr{sites.google.com/view/heathpearson}


\begin{document}

\begin{abstract} We define an intrinsic invariant \(\gamma_{\mathrm{Cox}}(\cdot)\), the \emph{maximum Cox complexity}, of Fano-type varieties \(X\). We prove \(\gamma_{\mathrm{Cox}}(X)\in\QQ_{\ge0}\) and \(\gamma_{\mathrm{Cox}}(X)=0\) if and only if \(X\) is toric. We then apply this invariant to the Mukai conjecture, which characterises powers of projective spaces among smooth Fano varieties in terms of \(\dim(X)\), the Picard number \(\rho(X)\), and the Fano index \(i(X)\).

All klt Fano varieties \(X\) canonically embed into a projective toric variety \(Z\) via their Cox ring. When \(Z\) is smooth, we call \(X\) \emph{ambiently smooth}. We prove the Mukai conjecture for ambiently smooth Fano varieties by establishing the stronger inequality:
\[
\dim(X)+\rho(X)-i(X)\rho(X)\ge \gamma_{\mathrm{Cox}}(X)\ge0.
\]

Notably, \(45\%\) of the deformation families of smooth Fano threefolds have an ambiently smooth member.

For a Fano-type variety \(X\), the invariant \(\gamma_{\mathrm{Cox}}(X)\) is defined using coefficient sums of anticanonical divisors supported on the divisors corresponding to minimal homogeneous generators of the Cox ring. Our proof proceeds by reducing the Cox ring to positive characteristic, using strong \(F\)-regularity to obtain the coefficients of an anticanonical divisor, and the \(F\)-pure threshold to establish the nonnegativity of \(\gamma_{\mathrm{Cox}}(X)\).
\end{abstract}

\maketitle

\section{Introduction}\label{introduction}

The Mukai conjecture (\cite{MukaiConjecture} 1988) proposes a numerical characterisation of powers of projective spaces among smooth Fano varieties. It has been researched extensively, including recent progress from various perspectives (see for example \cite{Wisniewski1990Mukai,Bonavero2003ConjectureMukai,Andreatta2004GeneralizedMukai,Fujita2016AroundMukai,Fujita2019GeneralizedMukaiToric,Reineke2024MukaiQuiver,NefComplexity,SphericalGeneralisedMukai}). Here, we define a new invariant \(\gamma_{\mathrm{Cox}}(\cdot)\) for Fano-type varieties \(X\). We prove its non-negativity \(\gamma_{\mathrm{Cox}}(X)\in\QQ_{\ge0}\), and prove  \(\gamma_{\mathrm{Cox}}(X)=0\) if and only if \(X\) is toric. Then we use this to prove the Mukai conjecture for a class of locally factorial Fano varieties  naturally defined in terms of their Cox rings (see \Cref{ambiently smooth}).

\begin{conje}[The Mukai conjecture {\cite[Conjecture~4]{MukaiConjecture}}]\label{mukai conjecture} Let \(X\) be a smooth Fano variety with Picard number \(\rho(X)\coloneq\rk\Pic(X)\) and Fano index \(i(X)\coloneq\max\{k\in\ZZ_{>0}\mid kH=[-K_X],\, H\in\Pic(X)\}\), then
\[
\mathrm{dim}(X)+\rho(X)-i(X)\rho(X)\ge0,
\]
with \(\mathrm{dim}(X)+\rho(X)-i(X)\rho(X)=0\) if and only if \(X\cong(\PP^{i(X)-1})^{\rho(X)}\) is a power of projective spaces.
\end{conje}

We now recall some historical developments which motivate our approach. The Mukai conjecture generalises the Kobayashi-Ochiai theorem (\cite{KobayashiOchiai1973} 1973), which proves that for smooth Fano varieties \(X\), the \emph{coindex} 
\[
\mathrm{coind}(X)\coloneq \mathrm{dim}(X)+1-i(X)
\]
is non-negative \(\mathrm{coind}(X)\ge0\). The coindex stratifies the class of smooth Fano varieties, and the strata with small coindex have been shown to have a simpler structure. For example: \(\mathrm{coind}(X)=0\) if and only if \(X\cong\PP^{i(X)-1}\) \cite{KobayashiOchiai1973}; \(\mathrm{coind}(X)=1\) if and only if \(X\cong Q\) is a quadric hypersurface \cite{KobayashiOchiai1973}; and the cases of \(\mathrm{coind}(X)=2,3\) are classified under the names of the \emph{Del Pezzo manifolds} \cite{Fujita1990}, and the \emph{Mukai manifolds} \cite{Mukai1989,Mella1999} respectively. Similarly, \Cref{mukai conjecture} identifies an invariant of smooth Fano varieties \(X\),
\[
\mu(X)\coloneq\mathrm{dim}(X)+\rho(X)-i(X)\rho(X)
\]
which is expected to be non-negative \(\mu(X)\ge0\), and satisfy \(\mu(X)=0\) if and only if \(X\cong(\PP^{i(X)-1})^{\rho(X)}\). So far, classification in the stratum \(\mu(X)=1\) has been completed in the cases \(\mathrm{dim}(X)\le5\) or \(\rho(X)\le 3\) \cite{Fujita2016AroundMukai}. Following the naming used in \cite{GMMS2023}, we call the quantity \(\mu(X)\) the \emph{Mukai complexity}.

In 2003, the \emph{generalised Mukai conjecture} was introduced \cite{Bonavero2003ConjectureMukai} to approach the Mukai conjecture using the geometry of rational curves on Fano varieties. The generalised version of \Cref{mukai conjecture} has the same statement but with the Fano index \(i(X)\) replaced by the \emph{pseudo-index} \(\iota(X)\coloneq\min\{(-K_X\cdot\gamma)\mid \gamma\text{ rational curve in } X\}\). So now we consider the quantity
\[
\mu_\iota(X)\coloneq\mathrm{dim}(X)+\rho(X)-\iota(X)\rho(X)
\]
and notice that since \(i(X)\mid \iota(X)\) we have \(\mu(X)\ge\mu_\iota(X)\).

\begin{conje}[The generalised Mukai conjecture {\cite[Conjecture]{Bonavero2003ConjectureMukai}}] Let \(X\) be a smooth Fano variety. Then 
\[
\mu_\iota(X)\ge0,
\]
with \(\mu_\iota(X)=0\) if and only if \(X\cong(\PP^{\iota(X)-1})^{\rho(X)}\).
\end{conje}

The generalised Mukai conjecture has been proved in various special cases, including for certain non-smooth Fano varieties admitting a group action, especially singular toric varieties \cite{Casagrande2006,Fujita2019GeneralizedMukaiToric}. A core motivation for us is the recent result of \cite{SphericalGeneralisedMukai}, where the generalised Mukai conjecture was proved for the \(\QQ\)-factorial spherical Fano varieties by working with their invariant rational curves. In this setting an intermediate bound was proved.
\begin{theor}[Theorem~3.8 {\cite{SphericalGeneralisedMukai}}]\label{spherical case} Let \(X\) be a \(\QQ\)-factorial spherical Fano variety, then
\[
\mu_\iota(X)\ge\gamma(X)\ge0,
\]
where
\[
\gamma(X)\coloneq\inf\left\{\gamma(X,\sum_{i}a_iS_i)\coloneq \dim(X)+\rk\Cl(X)-\sum_ia_i \Biggm| (X,\sum_i a_iS_i)\text{ Calabi-Yau pair}\right\}
\]
is the infimum \emph{absolute complexity \(\gamma(X,\Delta)\) of a Calabi-Yau pair} \((X,\Delta)\) (see \cite[Section~1]{SphericalGeneralisedMukai}). Moreover, \(\mu_\iota(X)=0\) if and only if \(X\cong(\PP^{\iota(X)-1})^{\rho(X)}\).
\end{theor}

The intermediate quantity \(\gamma(\cdot)\) in \Cref{spherical case} is defined in a more general setting, where it acts as a measure of toricness.

\begin{theor}[{\cite[Corollary~of~Theorem~1.2]{Geomcharoftoric}}]\label{absolute complexity} Let \(X\) be a normal proper variety with finitely generated class group, and \((X,\sum_ia_iS_i)\) a log canonical pair with \(-K_X\cong\sum_i a_iS_i\), then for \(\gamma(\cdot)\) as defined in \Cref{spherical case},
\[
\gamma(X)\ge0,
\]
moreover \(\gamma(X)\notin(0,1)\), and \(\gamma(X)=0\) if and only if \(X\) is a toric variety.
\end{theor}

Motivated by this development, we introduce an intrinsic numerical quantity \(\gamma_{\mathrm{Cox}}(\cdot)\) for Fano-type varieties (recall that a Fano-type variety is a normal variety \(X\) such that there exists an effective \(\QQ\)-divisor \(\Delta\) with \(-(K_X+\Delta)\) ample \(\QQ\)-Cartier, and \((X,\Delta)\) klt). In practice, the quantity \(\gamma_{\mathrm{Cox}}(\cdot)\) is computable by a linear program. We establish its non-negativity, and prove that the toric varieties are characterised as the varieties where this quantity equals zero.

\begin{definition}(Maximum Cox complexity) Let \(X\) be a Fano-type variety (so that \(\mathrm{Cox}(X)\) is finitely generated), and let \(D_1,\dots,D_m\) be the divisors associated to a minimal system of \(\Cl(X)\)-homogeneous generators of the Cox ring of \(X\) (whose divisor classes are uniquely determined, see \Cref{background section}). Define
\[
\gamma_{\mathrm{Cox}}(X)\coloneq \max\left\{\dim(X)+\rk\Cl(X)-\sum_{i=1}^m a_i\Biggm|-K_X\sim_\QQ \sum_{i=1}^m a_iD_i,\, 0\le a_i\le 1\right\}
\]
to be the \emph{maximum Cox complexity} of \(X\).
\end{definition}

\begin{remark}
The maximum Cox complexity \(\gamma_{\mathrm{Cox}}(X)\) is the maximum absolute complexity \(\gamma(X,\sum_{i=1}^ma_iD_i)\coloneq \mathrm{dim}(X)+\rk\Cl(X)-\sum_{i=1}^m a_i\) where \(\sum_{i=1}^ma_iD_i\) is an anticanonical boundary divisor supported on the Cox divisors. See \cite[Definition~1.7]{Geomcharoftoric} for the original definition of the absolute complexity in the setting of log pairs.
\end{remark}

\begin{theor}\label{toricness criterion}
Let \(X\) be a Fano-type variety. Then
\[
\gamma_{\mathrm{Cox}}(X)\ge0,
\]
moreover \(\gamma_{\mathrm{Cox}}(X)\notin(0,1)\), and \(\gamma_{\mathrm{Cox}}(X)=0\) if and only if \(X\) is a toric variety.
\end{theor}

We will connect the maximum Cox complexity \(\gamma_{\mathrm{Cox}}(X)\) to the Mukai complexity \(\mu(X)\). 

Firstly, since klt Fano varieties are Mori dream spaces, their Cox rings \(R\coloneq\mathrm{Cox}(X)\coloneq\bigoplus_{[D]\in\Cl(X)}H^0(X,\OO_X([D]))\) are finitely generated. The quasitorus \(H_X\coloneq\mathrm{Spec}\,\CC[\Cl(X)]\) acts on \(\mathrm{Spec}\,R\), and outside of an irrelevant locus \(V\subset\mathrm{Spec}\,R\) there is a good quotient by \(H_X\) which realises \(X\cong ((\mathrm{Spec}\,R)\setminus V)\sslash H_X\).
Fix a minimal system of \(\Cl(X)\)-homogeneous generators of \(R\) to get a presentation \(R=\CC[f_1,\dots,f_m]/I\). So \(\mathrm{Spec}\, R\subseteq\CC^m\), and there is a possibly larger irrelevant locus \(V'\supseteq V\), and a toric variety \(Z\coloneq (\CC^m\setminus V')\sslash H_X\) with \(\mathrm{Cox}(Z)=\CC[f_1,\dots,f_m]\). This construction gives us a closed embedding \(i: X\hookrightarrow Z\), where the torus invariant prime divisors \(T_1,\dots, T_m\) of \(Z\) restrict to prime divisors \(D_i\coloneq T_i|_X\subset X\) which we call the \emph{Cox divisors}. This induces an isomorphism of class groups \(i^*:\Cl(Z)\to \Cl(X)\) determined by \(i^*[T_i]=[D_i]\) for \(i=1,\dots,m\). Importantly, the multi-set of divisor classes of the Cox divisors is uniquely determined by the isomorphism class of \(X\). Furthermore, when \(X\) is Fano, there is a canonical choice for the ambient toric variety \(Z\), which is a projective variety polarised by \(\mathcal{L}\coloneq (i^*)^{-1}[-K_X]\) (see \Cref{background section} for details).
\\\\
We believe it is natural, given the canonicity of this construction, to stratify the class of klt Fano varieties depending on the singularity type of the toric ambient variety \(Z\) into which they can be canonically embedded, so we introduce:
\begin{definition}\label{ambiently smooth} An \emph{ambiently smooth Fano variety} is a klt Fano variety \(X\) such that its canonical toric ambient variety \(Z\) is smooth. 
\end{definition}

\begin{example}
Some well-known examples of ambiently smooth Fano varieties are generalised flag varieties, \(\dim\ge3\) smooth Fano complete intersections in projective space, and smooth toric Fano varieties. See \Cref{section on ambiently smooth} for further examples.

How common are ambiently smooth Fano varieties among all smooth Fano varieties?

\begin{table}[H]
    \centering
    \begin{tabular}{|c|c|c|c|c|c|}
        \hline
        Fano $n$-folds
        & \makecell[c]{Total}
        & \makecell[c]{Generalised\\flag}
        & \makecell[c]{Complete intersection\\in projective space}
        & Toric
        & \makecell[c]{Ambiently\\smooth} \\
        \hline
        \(n=3\)& $105$ & $5$ & $7$ & $18$ & $47$ \\
        \hline
    \end{tabular}
    \vspace{0cm}
    \caption{The number of deformation families of smooth Fano threefolds
    containing a member of given type.}
    \label{tab:fano-threefold-families}
    \vspace{-0.2cm}
\end{table}
So \(45\%\) of the deformation families of smooth Fano \(3\)-folds have an ambiently smooth member (Cox ring data available from \cite{BatyrevPopov2004,Suess2014,DHHKL2015,Suess2018}). We think the ambiently smooth Fano varieties are a natural and substantial class worth further study.
\end{example}

In this paper we use \Cref{toricness criterion} to prove the Mukai conjecture for the ambiently smooth Fano varieties, establishing a stronger intermediate bound as in \Cref{spherical case}.

\begin{theor}\label{main thm} Let \(X\) be an ambiently smooth Fano variety, then
\[
\mu(X)\ge \gamma_{\mathrm{Cox}}(X)\ge0,
\]
with \(\mu(X)=0\) if and only if \(X\cong(\PP^{i(X)-1})^{\rho(X)}\) is a product of projective spaces.
\end{theor}

    Here is a high level overview of the approach. Let \(X\) be a Fano-type variety. In \Cref{low coefficient sum} we construct an anticanonical boundary divisor \(\sum_{i=1}^ma_i D_i\sim_\QQ -K_X\) supported on the Cox divisors and prove that its sum of coefficients is bounded by \(\sum_{i=1}^ma_i\le\mathrm{dim}(X)+\rho(X)\). Additionally, we prove that the equality \(\sum_{i=1}^ma_i=\mathrm{dim}(X)+\rho(X)\) holds if and only if \(\mathrm{Cox}(X)\) is a polynomial ring, or equivalently \(X\) is toric. To do this, we pass to a positive characteristic reduction \(A\) of the Cox ring respecting the \(\Cl(X)\)-grading and anticanonical module. From here, the coefficients \(a_i\) are determined by strong \(F\)-regularity applied to the product of the generators \(f_1\cdots f_m\) of the Cox ring, and the upper bound on the coefficient sum \(\sum_{i=1}^m a_i\le\mathrm{fpt}(A_\mathfrak{m},\mathfrak{m}A_\mathfrak{m})\le \mathrm{dim}(X)+\rho(X)\) comes from the \(F\)-pure threshold with \(\mathfrak{m}=(f_1,\dots,f_m)\) the ideal of the vertex of \(\mathrm{Spec}\,R\), see \Cref{alg} and \Cref{low coefficient sum}. Once \Cref{toricness criterion} is proved, we look at the Mukai conjecture.

    \Cref{low coefficient sum} implies \(\gamma_{\mathrm{Cox}}(X)\ge0\). Now assuming \(X\) is ambiently smooth, in \Cref{Fano index} and \Cref{inequality} we use the geometry of the smooth ambient variety \(Z\) to prove that \(\mu(X)\ge\gamma_{\mathrm{Cox}}(X)\). Then in \Cref{equality}, we prove that \(\mu(X)=0\) characterises products of projective spaces.

\begin{remark}
    The similarity between \Cref{absolute complexity} and \Cref{toricness criterion} leads us to suspect that for the divisor \(D\coloneq\sum_{i=1}^ma_iD_i\) as constructed in the proof of \Cref{low coefficient sum}, that \((X,D)\) is log canonical. If this is indeed the case, then the non-negativity and toric characterisation in \Cref{toricness criterion} can also be proved using \Cref{absolute complexity}.
\end{remark}

\begin{remark} We consider how general the hierarchy \(\mu(X)\ge\gamma_{\mathrm{Cox}}(X)\) could be. Firstly, note that there exist smooth Fano varieties \(X\) such that for any expression \([-K_X]=\sum_i a_iS_i\) with \(a_i>0\) and \(S_i>0\) \(\ZZ\)-Weil, we have \(i(X)\rho(X)>\sum_i a_i\), so for these varieties no such inequality exists. For an example of such a variety, consider a del Pezzo surface of degree \(1\) (c.f. \cite[Example~3.9]{SphericalGeneralisedMukai}). This shows that the bound proved in \Cref{main thm} for the ambiently smooth Fano varieties (and also the absolute complexity version \(\mu_{\iota}(X)\ge\gamma(X)\) in the case of \(\QQ\)-factorial spherical varieties \cite{SphericalGeneralisedMukai}) is a consequence of the specialised geometry of these varieties.

    However, in \cite[Question~3.11]{SphericalGeneralisedMukai} a version of this hierarchy is stated which is expected to hold in general. We consider the \emph{infimum complexity of a Calabi-Yau pair} on \(X\). Define
    \begin{align*}
    c(X)\coloneq \inf\Biggl\{ \dim(X)+r(\sum_ia_iS_i)-\sum_ia_i\Biggm|\sum_ia_iS_i\le \Delta\sim_\QQ -K_X,\\
    a_i\in\QQ_{>0},\,S_i>0\,\,\ZZ\text{-Weil},\,(X, \Delta)&\,\text{log canonical}\Biggr\},
    \end{align*}
    where \(r(\sum_ia_iS_i)\coloneq\dim\mathrm{span}\{[S_i]\in\Cl(X)_\QQ\}\). By \cite[Theorem~1.2]{Geomcharoftoric} we have \(c(X)\ge0\) and as in \cite[Question~3.11]{SphericalGeneralisedMukai} we believe the following conjectural bound, which would prove the generalised Mukai conjecture.
\end{remark}
    \begin{conje}[{\cite[Question~3.11]{SphericalGeneralisedMukai}}]\label{hierarchy} Let \(X\) be a smooth Fano variety, then
    \[
        \mu_\iota(X)\ge c(X).
    \]
\end{conje}

Simultaneously and independently, an unreleased \mbox{OpenAI} model claimed a solution of the (generalised) Mukai conjecture using Gromov-Witten invariants \cite{OAI:The-Generalized-Mukai-Conjecture-September-24-2026}. Our approach appears genuinely distinct, and we are hopeful that these two approaches can be consolidated leading us to discover new connections. We also highlight that \cite{OAI:The-Generalized-Mukai-Conjecture-September-24-2026} concerns non-singular Fano varieties, and here we prove the Mukai conjecture for certain locally factorial Fano varieties.

\subsection{Further context and connections} We mention some recent directions related to the Mukai conjecture and numerical invariants which detect geometry.
\subsubsection{The nef complexity}
The recent proof of the \emph{Mukai-type conjecture} \cite{NefComplexity,Gongyo2026} establishes a complexity which measures how far a klt Fano variety is from being a product of projective spaces. It may be naturally viewed as a generalisation of the Kobayashi-Ochiai theorem.

\begin{definition}[Nef complexity {\cite[Definition~1.2]{NefComplexity}}] Let \(X\) be a Fano variety. Define
\begin{align*}
c_{\mathrm{nef}}(X)\coloneq\min\Biggl\{\mathrm{dim}(X)+\rho(X)-&\sum_{j=1}^k b_j\Biggm|\\
&\sum_{j=1}^k b_jS_j\sim_\QQ-K_X,\,S_j\not\equiv0\text{ nef }\ZZ\text{-Cartier},\, b_j\in\mathbb{Q}_{\ge0}\Biggr\}
\end{align*}
to be the \emph{nef complexity} of \(X\).
\end{definition}

\begin{theor}[{\cite[Corollary~1.5]{NefComplexity}}] Let \(X\) be a klt Fano variety, then
\[
c_{\mathrm{nef}}(X)\ge 0,
\]
with \(c_{\mathrm{nef}}(X)=0\) if and only if \(X\cong\PP^{n_1}\times\dots\PP^{n_{\rho(X)}}\) is a product of projective spaces.
\end{theor}

\begin{remark}
Note there is not a simple relationship between the nef complexity \(c_{\mathrm{nef}}(X)\) and the Mukai complexity \(\mu(X)\) in the sense that \(c_{\mathrm{nef}}(X)\not\ge \mu (X)\), and \(\mu (X)\not\ge c_{\mathrm{nef}}(X)\) \cite[Examples~3.2 and~3.3]{Gongyo2026}.
\end{remark}

\subsubsection{A combinatorial smoothness criterion for spherical varieties}
For \(X\) as in \Cref{toricness criterion}, \(\gamma_{\mathrm{Cox}}(X)\) is closely related to the effective criteria in \cite[Theorem~1.3]{GHP2026} to detect the smoothness of an arbitrary spherical variety along its \(G\)-orbits, and to detect the toricness of an arbitrary complete spherical variety. Indeed, in \cite{GHP2026} a linear program \(\widetilde{\wp}(\cdot)\) is introduced with an analogous region of optimisation and objective functional to \(\gamma_{\mathrm{Cox}}(\cdot)\). Instead of taking place on the level of Cox ring data, it is computable from the defining combinatorial data of spherical varieties. Its proof originates from the proof of the generalised Mukai conjecture for spherical varieties (\Cref{spherical case}).

\subsection{The sum of lengths of extremal rays minus one}
Related to the Mukai conjecture, K. Fujita introduced an invariant \(s(X)\) of smooth Fano varieties.
\begin{definition}[{\cite[Defined~in~Question~1.3]{Fujita2014}}]\label{s definition} Let \(X\) be a smooth Fano variety. Define
\[
s(X)\coloneq\sum_{R\text{ extremal ray}}(\ell(R)-1),
\]
where \(\ell(R)\coloneq\min\{-K_X\cdot C\mid C\text{ rational curve in \(X\)},\,[C]\in R\}\) is the length of the extremal ray \(R\), and where \(R\) ranges over all extremal rays of the Mori cone in the sum above.
\end{definition}

Observe that we have \(\mu_{\iota}(X)\ge \mathrm{dim}(X)- s(X)\), so the inequality in the generalised Mukai conjecture would be implied by \(\mathrm{dim}(X)-s(X)\ge0\). In fact, when all extremal rays of \(X\) are of fibre type the generalised Mukai conjecture is proved in \cite{Fujita2014} in this way. Although in general, there are examples known with \(\mathrm{dim}(X)-s(X)<0\). In \Cref{s section} we prove that among the toric varieties \(s(X)\) admits no upper bound which is polynomial in \(\mathrm{dim}(X)\).

\subsection{Acknowledgements} I thank Johannes Hofscheier for his helpful comments, as well as for his coaching as my advisor. I thank Kento Fujita for helpful conversations, and for his hospitality during my visit to Japan.

This work was partially completed while the author was a JSPS Summer International Research Fellow under the JSPS Summer Program 2026.

\subsection{AI use declaration} A previous version of this document (arXiv:2604.25023v1) proves the Mukai conjecture for the ambiently smooth Fano varieties, subject to additional assumptions on the Cox ring. All aspects of that version are AI-free. After a brief discussion with GPT-6 Astra, the old version of \Cref{low coefficient sum} with the same statement but relying on additional assumptions was relaxed to the current version of \Cref{low coefficient sum} which allows the additional assumptions on the Cox ring to be lifted. This means the proof now follows for the full class of varieties of \Cref{ambiently smooth}.

The same tool has been used to find references for results in the literature.

All writing is my own.
\\\\
Write \(n\coloneq \mathrm{dim}(X)\). Unless otherwise stated, varieties will be defined over \(\mathbb{C}\). 

\section{Background}\label{background section}

We introduce the necessary theory of Cox rings used in this paper. In the first part, we introduce the convex geometrical data of the Cox ring which defines the canonical ambient toric variety, which is used to define the ambiently smooth Fano varieties of \Cref{ambiently smooth}, and in Propositions~\ref{Fano index},~\ref{inequality},~\ref{equality}. In the second part, we see that some well-known families of Fano varieties are ambiently smooth. In the final part, we introduce the necessary algebraic data of the Cox ring for the reduction to positive characteristic in \Cref{low coefficient sum}.

\subsection{The canonical toric embedding}\label{canonical toric intro}

Let \(X\) be an irreducible, normal variety with finitely generated class group \(\Cl(X)\) (for example, let \(X\) be klt Fano type variety \cite{BCHM2010}). Then the \emph{Cox ring} of \(X\) is the \(\Cl(X)\)-graded ring

\[
\mathrm{Cox}(X)\coloneq\bigoplus_{[D]\in\Cl(X)} H^0(X,\OO(D)),
\]
whose definition is made precise in \cite{CoxRingBook}. If \(\mathrm{Cox}(X)\) is also finitely generated and \(X\) is projective, then \(X\) is called a \emph{Mori dream space}. In this case, there is a minimal system of pairwise non-associated \(\Cl(X)\)-prime generators of \(\mathrm{Cox}(X)\) whose multi-set of \(\Cl(X)\)-degrees is uniquely determined up automorphism of \(\Cl(X)\) \cite[Lemma~2.5]{CoxRingBook}, so we obtain a presentation:
\[
\mathrm{Cox}(X)=\frac{\CC[f_1,\dots,f_m]}{I}.
\]

Importantly, the isomorphism class of \(X\) determines the \(\Cl(X)\)-graded polynomial ring \(\CC[f_1,\dots,f_m]\), which is the graded Cox ring of a toric variety, and up to isomorphism finitely many toric varieties share this graded polynomial Cox ring. Furthermore, each Mori dream space \(X\) sharing the Cox ring \(\mathrm{Cox}(X)\) admits a canonical closed embedding \(i:X\hookrightarrow Z\) into at least one of these toric varieties \(Z\), and there is an isomorphism \(i^*:\Cl(Z)\to\Cl(X)\). Moreover, when \(X\) is Fano, the isomorphism class of the ambient variety \(Z\) is uniquely determined from the isomorphism class of \(X\).

We now introduce this theory. We assume the reader is familiar with the well-known characterisation of toric varieties in terms of lattice fans. Firstly, let us fix some notation. For a \(d\)-dimensional toric variety \(Z\), let \(N\coloneq\Hom(\CC^*,(\CC^*)^d)\), and let \(\Sigma\) be its fan in \(N_\QQ\cong\QQ^d\). For any \(k=1,\dots,d\) write \(\Sigma^{(k)}\) for the set of \(k\)-dimensional cones in \(\Sigma\), for example \(\Sigma^{(1)}\) is the set of rays of \(\Sigma\). Next, for any subset of cones \(S\subseteq\Sigma\), we write \(\rho(S)\) for the set of primitive ray generators of the rays of the cones in \(S\). We always assume that \(\mathrm{span}_\QQ(\Sigma)=N_\QQ\). 

To begin, consider the following exact sequence of \(\ZZ\)-modules, where \(m=d+\rk\Cl(Z)\) is the number of torus invariant prime divisors \(\{T_1,\dots,T_m\}\) of \(Z\),
\[\begin{tikzcd}
	0 & L & \ZZ^m & N,
	\arrow[from=1-1, to=1-2]
	\arrow[from=1-2, to=1-3]
	\arrow["P", from=1-3, to=1-4]
\end{tikzcd}\]
and its dual exact sequence below
\[\begin{tikzcd}
	0 & K & (\ZZ^m)^* & M & 0.
	\arrow[from=1-2, to=1-1]
	\arrow["Q"', from=1-3, to=1-2]
	\arrow[from=1-4, to=1-3]
	\arrow[from=1-5, to=1-4]\tag{\(\downarrow\)}
\end{tikzcd}\]

Let \(\{e_i\mid i=1,\dots,m\}\) be the standard basis of \(\ZZ^m\), and let the corresponding dual basis of \((\ZZ^m)^*\) be \(\{e_i^*\mid i=1,\dots,m\}\). Here, the \(i\)th coordinate corresponds to the \(i\)th torus invariant prime divisor \(T_i\) in the following way: we have \(P\colon\ZZ^m\to N\) is a choice of toric generator matrix, with \(P(e_i)\) the coordinate of the \(i\)th ray of \(\Sigma\), and \(K\cong\Cl(Z)\), with \(Q(e_i^*)\) equal to the divisor class of \(T_i\). We say the vector configuration \(\{Q(e_i^*)\mid i=1,\dots,m\}\) is \emph{Gale dual} to the vector configuration \(\{P(e_i)\mid i=1,\dots,m\}\). For shorthand, if \(v=P(e_i)\) we write \(v^\downarrow=Q(e_i^*)\) for its \emph{Gale dual}. Note that the Gale dual of a vector configuration is defined up to invertible change of coordinates, and the rational version is defined analogously.

We now introduce the Gale dual version of a lattice fan. Since each \(\sigma\in\Sigma\) is of the form \(\sigma=\cone(P(e_i)\mid i\in I(\sigma))\) for some \(I(\sigma)\subseteq\{1,\dots,m\}\), we may define the multi-set
\[
\sigma^\downarrow\coloneq (Q(e_i^*)\mid i\notin I(\sigma)),
\]
and the set \(\Sigma^\downarrow\coloneq\{\sigma^\downarrow\mid\sigma\in\Sigma\}\). Now \(\mathrm{Cox}(Z)=\CC[f_1,\dots,f_m]\) is graded by \(\Cl(Z)\) with \(\deg(f_i)=Q(e_i^*)\) for \(i=1,\dots,m\), and \((\mathrm{Cox}(Z),\Sigma^\downarrow)\) is called the \emph{bunched polynomial ring} of \(Z\), from which we can recover the fan \(\Sigma\) determining the isomorphism class of \(Z\). 

We now describe the Mori dream spaces \(X\) which are Fano in terms of \emph{bunched rings}. Suppose \(X\) has Cox ring presentation \(\CC[f_1,\dots,f_m]/I\). Then the movable cone of \(X\) is

\[
\mathrm{Mov}(R)\coloneq\bigcap_{j\in\{1,\dots,m\}}\cone(Q(e_i^*)\mid i\in \{1,\dots,m\}\setminus\{j\})\subset K_\QQ.
\]

Since \(X\) is Fano, it is polarised by \(\mathcal{L}\coloneq[-K_X]\in\mathrm{Mov}(R)^\circ\cap\Pic(X)_\QQ\). From this, the isomorphism class of \(X\) is determined by the \emph{bunched ring} \((\CC[f_1,\dots,f_m]/I,\Phi(\mathcal{L}))\), where
\[
\Phi(\mathcal{L})\coloneq\left\{\tau_J\coloneq(Q(e_j^*)\mid j\in J)\,\Bigg{|}\,\mathcal{L}\in\cone(\tau_J)^\circ,\, \,\prod_{j\in J} f_j\notin\mathrm{rad}\langle f_j\mid j\notin J\rangle\right\}.
\]

From this data we may also canonically obtain the bunched polynomial ring of a toric variety \(Z\) (as in \cite[Construction~3.2.5.7]{CoxRingBook}) as \((\CC[f_1,\dots,f_m],\Sigma(\mathcal{L})^\downarrow)\), where
\[
\Sigma(\mathcal{L})^\downarrow\coloneq\left\{\tau_J\coloneq(Q(e_j^*)\mid j\in J)\mid \mathcal{L}\in\cone(\tau_J)^\circ\right\}\supseteq\Phi(\mathcal{L}).
\]

We now state some important facts arising from this construction. Firstly, there is a closed embedding \(X\hookrightarrow Z\), called the \emph{canonical toric embedding} of \((\CC[f_1,\dots,f_m]/I,\Phi(\mathcal{L}))\).

\begin{remark}
Remark that \(Z\) is a projective toric variety. In the non-Fano setting while there remains a canonical toric ambient embedding (as in \cite[Construction~3.2.5.3]{CoxRingBook}), it is not necessarily projective, yet it always admits a possibly non-canonical projective completion. Furthermore,  the theory of bunched rings can be extended to classify not necessarily complete varieties with finitely generated Cox ring (see \cite[Section~3.2]{CoxRingBook}).
\end{remark}

There is a finite set of prime divisors \(\{D_1,\dots, D_m\}\) in \(X\) such that \(T_i\cap X=D_i\) for each \(i=1,\dots,m\). These divisors give us an isomorphism \(\Cl(Z)\cong\Cl(X)\) via \([T_i]\mapsto [D_i]=\deg(f_i)\) \cite[Proposition~3.2.5.4]{CoxRingBook}. So we are able to conflate the notation, writing \([D]=(i^*)^{-1}[D]\in\Cl(Z)\) for \([D]\in \Cl(X)\).
Next, we have an explicit description of the Picard group \cite[Corollary~3.3.1.6]{CoxRingBook}: \(\Pic(X)=\bigcap_{\tau\in\Phi(\mathcal{L})}\langle\tau\rangle_\ZZ\),
where \(\langle\cdot\rangle_\ZZ\) denotes the \(\ZZ\)-span.
We have that \(X\) is locally factorial if and only if for all \(\tau\in\Phi(\mathcal{L})\), we have that \(\langle\tau\rangle_\ZZ\) generates \(K\) \cite[Corollary~3.3.1.8]{CoxRingBook}. 

We also recall some well-known elements of the theory of toric varieties (see \cite{toricvarieties}). Let \((Z,\mathcal{L})\) be a polarised toric variety, where \(\mathcal{L}=\sum_{i=1}^ma_iT_i\) with \(a_i>0\) for all \(i=1,\dots,m\) is an ample torus invariant \(\QQ\)-Cartier \(\QQ\)-divisor. Then the \emph{moment polytope} of \((Z,\mathcal{L})\) is the following bounded intersection of half-spaces 
\[
\mathcal{P}(Z,\mathcal{L})\coloneq\bigcap_{i=1}^m\{y\in M_\QQ\mid\langle P(e_i),y\rangle+a_i\ge0\},
\]
where \(\langle\cdot,\cdot\rangle\colon N\times M\to \ZZ\) is the dual pairing. This polytope has normal fan \(\Sigma\). Let the vertices of \(\mathcal{P}(Z,\mathcal{L})\) be \(\{C_\sigma\mid \sigma\in\Sigma^{(d)}\}\), which we call the \emph{\(\QQ\)-Cartier data} of \(\mathcal{L}\); and which satisfy \(\langle P(e_i),C_\sigma\rangle+a_i\ge0\), with \(\langle P(e_i),C_\sigma\rangle+a_i=0\) if and only if \(P(e_i)\in\sigma\).

\subsection{Ambiently smooth Fano varieties}\label{section on ambiently smooth} In the process of generalising well-developed toric results into the setting of general Fano varieties, the ambiently smooth Fano varieties may be a suitable class to research as an intermediate step.

Following the construction of \Cref{canonical toric intro}, by definition, a Fano variety \(X\) with \(\mathcal{L}=[-K_X]\) is ambiently smooth as in \Cref{ambiently smooth} if and only if the projective toric variety \(Z\) with toric fan \(\Sigma(\mathcal{L})\) is a smooth variety. This definition is canonical, as the isomorphism class of \(Z\) depends only on the isomorphism class of \(X\). Remark that \(X\) itself is not necessarily smooth.

\begin{proposition}\label{fano} Ambiently smooth Fano varieties are locally factorial and have free class group.

\end{proposition}

\begin{proof} For all \(\tau\in\Phi(\mathcal{L})\subseteq\Sigma(\mathcal{L})^\downarrow\) by smoothness of \(Z\) we have \(\langle\tau\rangle_\ZZ\) generates \(\Cl(Z)\cong\Cl(X)\), so \(X\) is locally factorial by the criterion of \cite[Corollary~3.3.1.8]{CoxRingBook}. By ambient smoothness we have \(\Cl(X)\cong\Cl(Z)\cong\Pic(Z)\) which is free.
\end{proof}

The ambiently smooth Fano varieties simultaneously generalise the following three well-known classes of varieties.

\begin{example}[Generalised flag varieties] Let \(X= G/P\), where \(G\) is a connected reductive algebraic group and \(P\) a parabolic subgroup. Then using standard theory of spherical varieties, \(\Cl(X)=\bigoplus_{\alpha\in\D^b} \ZZ[D_\alpha]\) is freely generated. We have that the Cox ring of \(X\) has a minimal presentation of the form \(\mathrm{Cox}(X)=\CC[f_{D_\alpha^i}\mid \alpha\in \D^b,\,i=1,\dots,m_{d_\alpha}]/I\), with \(\deg(f_{D_\alpha^i})=e_\alpha\) for \(i=1,\dots,m_{d_\alpha}\) for some \(m_{d_\alpha}\in\ZZ_{\ge2}\) (e.g. see \cite{Gagliardi2014}). Since \(\Eff(X)\) is simplicial, all bunched cones supported on this data have support equal to the effective cone. Therefore all divisor classes in the ample cone are Mori equivalent, so we find the canonical ambient toric variety is \(Z=\PP^{m_{d_{\alpha_1}}-1}\times\dots\times\PP^{m_{d_{\alpha_s}}-1}\) a product of projective spaces. So generalised flag varieties are ambiently smooth.
\end{example}

\begin{example}[\(\dim\ge3\) smooth Fano complete intersections] Let \(X=\mathbb{V}(g_1,\dots,g_r)\subset\PP^N\) with \(\deg g_i=d_i>1\) be a smooth complete intersection with \(\mathrm{dim}(X)\ge3\) which is assumed to be Fano (\(\sum_i d_i< N+1\)). By the Grothendieck-Lefschetz theorem, \(\Cl(X)=\Pic(X)=\ZZ\,\OO_X(1)\). So by \cite[Remark~1.4.1.3]{CoxRingBook} \(\mathrm{Cox}(X)=\CC[f_1,\dots,f_{N+1}]/(g_1,\dots,g_r)\) with \(\deg(f_i)=1\) for \(i=1,\dots,N+1\). Therefore the canonical toric ambient variety is \(Z=\PP^{N}\). So \(X\) is ambiently smooth.
\end{example}

\begin{example}[Smooth toric Fano varieties] In this case \(X=Z\). Recall that locally factorial toric varieties are smooth.
\end{example}

We exhibit some further examples of ambiently smooth Fano varieties.

\begin{example}[\(X\) is non-smooth] Let \(X\coloneq \{X_0X_1+X_2X_3+X_4^2=0\}\subset\PP^5\). It has \(Z=\PP^5\), so \(X\) is ambiently smooth. There is an \(A_1\) singularity at \([0:0:0:0:0:1]\).
    
\end{example}

\begin{example}[\(Z\) is non-Fano] Consider \(X\) smooth Fano as in \cite[Theorem~1.1,~No.\,5,~\(a=0\),~\(m=1\)]{FHN2018}. We have \(\mathrm{Cox}(X)=\CC[f_1,f_2,f_3,g_1,g_2,g_3,h]/(f_1g_1^2+f_2g_2^2+g_3h)\), with \(\Cl(X)\cong\ZZ^2\), \(\deg(f_i)=e_1\), \(\deg(g_i)=e_2\), \(\deg(h)=e_1+e_2\), \(\deg(-K_X)=3e_1+2e_2\), \(\deg(-K_Z)=4e_1+4e_2\). The bunched cones of \(Z\) are the index sets of Cox ring generators such that the rational cone generated by their degrees contains \([-K_X]\) in their relative interior. Here are the inclusion minimal bunched cones of \(Z\):
\[
\Sigma(\mathcal{L})^\downarrow_{\min}=\{(1,4),(1,5),(1,6),(2,4),(2,5),(2,6),(3,4),(3,5),(3,6),(1,7),(2,7),(3,7)\}.
\]
We have \(|\det(Q(e_i^*)\mid i\in I)|=1\) for all \(I\in\Sigma(\mathcal{L})^\downarrow_{\min}\), so \(Z\) is smooth, so \(X\) is ambiently smooth.
We have \(\Nef(Z)=\bigcap_{I\in\Sigma(\mathcal{L})^\downarrow_{\min}}\cone(\{Q(e_i^*)\mid i\in I\})=\cone(e_1,e_1+e_2)\). So \([-K_Z]\notin\Nef(Z)^\circ=\mathrm{Ample}(Z)\), so \(Z\) is non-Fano. We have \(\mathrm{dim}(X)=4\), and that \(\sum_{j=1}^3(\frac{2}{3}D_{f_j}+\frac{1}{3}D_{g_j})+D_h\sim_\QQ[-K_X]\) computes \(\gamma_\mathrm{Cox}(X)=2\).

\end{example}

\subsection{Algebra background on Cox rings}\label{alg} Let \(X\) be a Fano-type variety, so there is an effective \(\QQ\)-divisor \(\Delta\) such that \((X,\Delta)\) is klt and \(-(K_X+\Delta)\) is ample. We now state some properties of the Cox ring which we require to prove \Cref{low coefficient sum}.


Since the effective cone is pointed and contains each \(\omega_i\coloneq\deg(f_i)\) for \(i=1,\dots,m\), the \(K\cong\Cl(X)\)-grading of \(R\) can be coarsened to a \(\ZZ\)-grading \(d:K\to \ZZ\) with \(d(w_i)>0\) for \(i=1,\dots,m\). Therefore \(\mathfrak{m}\coloneq (f_1,\dots,f_m)=\bigoplus_{\alpha\neq0}R_\alpha\subset R\)  is a maximal ideal. Write \(\overline{X}\coloneq\mathrm{Spec}\,R\), the Krull dimension of the Cox ring is \(\dim\overline{X}=\dim R=\mathrm{ht}(\mathfrak{m})=n+\rho(X)\).

Since \(X\) is \(\QQ\)-factorial and of Fano-type, we have that \(R\) is a normal Cohen-Macaulay domain, and \(\overline{X}\) is klt \cite[Theorem~1.2]{Brown2013}. 

To study the canonical class we look at the canonical module. Write \(\overline{X}=\mathrm{Spec}\,R\), and \(U\coloneq \overline{X}\setminus\overline{X}_\mathrm{sing}\), then the \emph{canonical module} is \(\omega_R\coloneq H^0\left(U,\bigwedge^{\dim R}\Omega_U^1\right)\), see \cite[Definition~3.1]{Canonicalmodule}. For \([D]\in\Cl(X)\), let \(R_{[D]}\coloneq H^0(X,\OO_X(D))\) with representatives chosen as in \cite[Construction~1.4.1.1]{CoxRingBook}. For \(\alpha\in \Cl(X)\) we write \(R(\alpha)\) for \(R\) with the shifted grading \(\deg_{R(\alpha)}\coloneq \deg_{R}-\alpha\), so that \(R(\alpha)_\beta= R_{\alpha+\beta}\) for all \(\beta\in\Cl(X)\). Now for \(L\coloneq [-K_X]\), we have the graded isomorphism \(\omega_R\cong R(-L)\) \cite[Theorem~2]{Braun2022}, so \(\omega_R\) is a free \(R\)-module of rank one with homogeneous generator of degree \(L\).

In the proof of \Cref{low coefficient sum}, we replace \(R\) by a positive characteristic reduction. We now introduce the procedure for doing this, see \cite[Section~2]{Takagi2013}. For a presentation \(R\cong\CC[f_1,\dots,f_m]/(g_1,\dots,g_r)\), let \(\Lambda\subset \CC\) be a finitely generated \(\ZZ\)-algebra containing all the coefficients of the \(g_j\), \(j=1,\dots,r\). Now \(\mathcal{R}\coloneq \Lambda[f_1,\dots,f_m]/(g_1,\dots,g_r)\)  satisfies \(\mathcal{R}\otimes_\Lambda\CC=R\) and is called a \(\Lambda\)-\emph{model} of \(R\). The inclusion \(\Lambda\to\mathcal{R}\) defines a morphism \(\pi:\mathrm{Spec}\,\mathcal{R}\to\mathrm{Spec}\,\Lambda\). Now for a closed point \(\mu\in\mathrm{Spec}\,\Lambda\) put \(k(\mu)\coloneq \Lambda/\mu\), then \(\mathrm{Spec}\,A_\mu\coloneq\pi^{-1}(\mu)\) and the reduction is \(A_\mu=\mathcal{R}\otimes_\Lambda k(\mu)\). Note that \(\mathrm{char}\, A_\mu=\mathrm{char}\,k(\mu)>0\) \cite[Section~2, p.~924]{Takagi2013}.

To replace \(R\) by a positive characteristic reduction while respecting the canonical module, we use the following description  \(\omega_R\cong \operatorname{Ext}^{m-\dim R}_S(R,S(-\sum_{i=1}^m w_i))\) \cite[Proposition~2.2.4]{GotoWatanabe1978II}.
Let us look into the computation of this formula. Firstly, since \(R\) is Cohen-Macaulay there exists a graded free resolution of length \(c\coloneq m-\dim R\) \cite[Theorem~1.3.3 and Proposition~1.5.15]{BrunsHerzog1998}, \(
0\longrightarrow F_c\xrightarrow{\partial_c}F_{c-1}\longrightarrow\cdots\longrightarrow F_0=S\longrightarrow R\longrightarrow0\). After fixing bases of the \(F_i\), the entries of \(\partial_i\) for \(i=1,\dots,c\) are finitely many polynomials in \(S=\CC[f_1,\dots,f_m]\). Then \(\partial_c^\vee:\operatorname{Hom}_S(F_{c-1},S(-\sum_{i=1}^mw_i))\to\operatorname{Hom}_S(F_c,S(-\sum_{i=1}^mw_i))\) gives us \(\omega_R\cong\mathrm{coker}\,\partial_c^\vee\). So for our \(\Lambda\)-model \(\mathcal{R}\) we choose \(\Lambda\subset\CC\) to be a finitely generated \(\ZZ\)-algebra which also contains all coefficients of the polynomials appearing as elements in \(\partial_i\), \(i=1,\dots,c\), as well as finitely many further coefficients to arrange \(\mathcal{R}(-L)\cong\mathrm{coker}(\partial_c^\vee)\). In \Cref{low coefficient sum} we also modify \(\Lambda\) to ensure the canonical module passes through this reduction step to get \(\omega_A\cong A(-L)\).

Let \(A\) be a reduced ring of \(\mathrm{char}\,p>0\). For some \(e\in\ZZ_{\ge1}\), let \(q\coloneq p^e\). Write \(A^{1/q}\coloneq\{a^{1/q}\mid a\in A\}\) for the ring of formal \(q\)th roots of \(A\). It is \(\frac{1}{q}K\)-graded with \(\deg(a^{1/q})=\deg(a)/q\). From now on we also assume \(A\) is a Noetherian domain, and we now recall some definitions concerning Frobenius splitting. Our reference is \cite{SmithZhang2015}. 

We say \(A\) is \emph{\(F\)-finite} if the Frobenius map \(F:a\mapsto a^p\) is finite. Useful for us is: when \(k\) is perfect every finitely generated \(k\)-algebra is \(F\)-finite \cite[p.~293]{SmithZhang2015}, and from now on we assume \(A\) is \(F\)-finite.

Let \(a\in A\) be in the complement of the minimal primes of \(A\) (when \(A\) is a domain, any \(a\neq0\) satisfies this property). We say \(A\) is \emph{strongly \(F\)-regular} if for any such \(a\in A\), there exists \(e\in\ZZ_{\ge1}\), \(q=p^e\), and an \(A\)-linear map \(\psi:A^{1/q}\to A\) with \(\psi(a^{1/q})=1\) \cite[Definition~1.13]{SmithZhang2015}.

We say \(A\) is \emph{\(F\)-pure} if there exists an \(A\)-linear map \(A^{1/p}\to A\), such that \(1\mapsto 1\). So observe that strong \(F\)-regularity immediately implies \(F\)-purity.

Let \((B,\mathfrak{n})\) be an \(F\)-finite \(F\)-pure local domain, and let \(s\in\RR_{\ge0}\).
The pair \((B,\mathfrak{n}^s)\) is called \emph{sharply \(F\)-pure} if for some \(e\in\ZZ_{\ge1}\), \(q=p^e\), there exists some \(b\in\mathfrak{n}^{\lceil s(q-1)\rceil}\) and a \(B\)-linear map \(\psi:B^{1/q}\to B\) with \(\psi(b^{1/q})=1\). Define \cite[Definition~2.2 and Remark~2.3(1)]{SinghTakagiVarbaro2017} 
\[
\mathrm{fpt}_B(\mathfrak{n})\coloneq\sup\{s\in\RR_{\ge0}\mid (B,\mathfrak{n}^s)\text{ is sharply \(F\)-pure}\}.
\]

When \(B\) is also Cohen-Macaulay, we have the following upper bound. Recall that a local domain \((B,\mathfrak{n})\) is regular if and only if \(\dim_k\mathfrak{n/\mathfrak{n}}^2=\mathrm{dim}\,B\).
\begin{proposition}[{\cite[Lemma~3.3]{SinghTakagiVarbaro2017}}]\label{fpt} Let \(B\) be an \(F\)-finite \(F\)-pure Cohen-Macaulay local domain, then \(\mathrm{fpt}_B(\mathfrak{n})\le\dim B\). Moreover if \(B\) is not regular, then \(\mathrm{fpt}_B(\mathfrak{n})\le \dim B -1\).
\end{proposition}

\section{The Mukai conjecture}\label{proof section}

In this section we prove \Cref{main thm}. We begin by selecting an effective \(\QQ\)-divisor \(\sum_{i=1}^ma_iD_i\) representing the anticanonical class in terms of the Cox divisors \(\{D_i\coloneq D_i^Z\cap X\mid i=1,\dots, m\}\), such that its coefficients and the sum of its coefficients are small. To do this we pass to a positive characteristic reduction \(A\) of the Cox ring. For \(X\) \(\QQ\)-factorial of Fano-type, its Cox ring is Gorenstein and klt, so in particular by \cite[Theorem~1.25(1)]{SmithZhang2015} we are able to obtain a reduction \(A\) which is strongly \(F\)-regular. Then the coefficients \(a_i\) are obtained by applying strong \(F\)-regularity to the product \(f_1\cdots f_m\) of the generators of \(A\) (of \(R\)), and the coefficient sum is bounded by \(\sum_{i=1}^m a_i\le\mathrm{fpt}(A_\mathfrak{m},\mathfrak{m}A_\mathfrak{m})\le n+\rho(X)\), for \(\mathfrak{m}=(f_1,\dots,f_m)\).

\begin{theor}\label{low coefficient sum} Let \(X\) be of Fano-type (there exists an effective \(\QQ\)-divisor \(\Delta\) with \(-(K_X+\Delta)\) ample and \((X,\Delta)\) klt). Let \(f_1,\dots,f_m\) be a minimal system of homogeneous non-units generating the Cox ring of \(X\). Let \(w_i\coloneq\deg(f_i)\). Then there exist rational numbers \(0<a_i\le 1\) for \(i=1,\dots,m\) such that 
\[
\sum_{i=1}^ma_iw_i=[-K_X],\quad\text{and}\quad\sum_{i=1}^m a_i\le \dim(X)+\rk\Cl(X).
\]
Furthermore, if \(\mathrm{Cox}(X)\) is not a polynomial ring (equivalently \(X\) is not toric) then the \(a_i\) can be chosen so that
\[
\sum_{i=1}^{m}a_i\le \dim(X)+\rk\Cl(X)-1.
\]
Since \(h^0(X,\OO_X(w_i))>0\) for all \(i=1,\dots,m\) we can equivalently write this in terms of the Cox ring divisors \(D_i\coloneq T_i|_X\) as
\[
D\coloneq\sum_{i=1}^ma_iD_i\sim_\QQ-K_X.
\]
\end{theor}

\begin{proof}
We construct a positive characteristic reduction \(A\) as in \Cref{alg} satisfying \(\dim A=\dim R\), preserving the \(K=\Cl(X)\)-grading of \(R\), and respecting the canonical module in the sense \(\omega_R\cong R(-L)\) and \(\omega_A\cong A(-L)\). We now use properties of \(R\) introduced in \Cref{alg}. We have \(R\) is Cohen-Macaulay and \(\omega_R\) is a free \(R\)-module of rank one, so \(R\) is Gorenstein. Since \(R\) is also klt, we have \(R\) is of strongly \(F\)-regular type \cite[Theorem~1.25(1)]{SmithZhang2015}: this means that 
\[
    \mathcal{F}\coloneq\{\mu\in\mathrm{Spec}\,\Lambda\mid \mu\text{ is closed},\, A_\mu\text{ is strongly }F\text{-regular}\}\subseteq\mathrm{Spec}\,\Lambda
\]
is Zariski dense in \(\mathrm{Spec}\,\Lambda\). 
By generic freeness, there exists \(0\neq a\in\Lambda\) such that after replacing \(\Lambda\) with \(\Lambda'\coloneq \Lambda[1/a]\) we have \(\mathcal{R}\) is \(\Lambda'\)-flat, and so that we have the element \(g\coloneq f_1\cdots f_m\neq0\) is nonzero in \(A_\mu\) \cite[Theorem~14.4]{Eisenbud1995}. By flatness, for every fibre \(A_\mu\) we have \(\dim A_\mu=\dim R=n+\rho(X)\), and the complex computing \(\omega_R\) in \Cref{alg} remains exact after tensoring by \(k(\mu)\) \cite[Section~6.2]{Eisenbud1995}. So there exists \(\mu'\in\mathrm{Spec}(\Lambda')\cap \mathcal{F}\subset\mathrm{Spec}(\Lambda)\) such that \(A_{\mu'}\) is strongly \(F\)-regular. Now we write \(A\coloneq A_{\mu'}\). Strong \(F\)-regularity implies \(A\) is normal and Cohen-Macaulay \cite[Theorem~1.18]{SmithZhang2015}, and we have \(k(\mu')\) is perfect \cite[Theorem~4.19]{Eisenbud1995} with \(p\coloneq\mathrm{char}\,k(\mu')\), so \(A\) is \(F\)-finite \cite[Lemma~2.4]{SchwedeTucker2012}. So we have a graded isomorphism
\(\omega_{A}\cong A(-L)\).

The isomorphism \(A\to A^{1/q},\,a\mapsto a^{1/q}\) preserves the canonical module, and we have \(\omega_{A^{1/q}}\cong A^{1/q}(-L/q)\). We have \(\Hom_A(A^{1/q},A)\) is an \(A^{1/q}\)-module with \((a\cdot\varphi)(x)=\varphi(ax)\) for any \(\varphi:A^{1/q}\to A\), \(a\in A^{1/q}\). We now show it is a free \(A^{1/q}\)-module, and compute the degree of its homogeneous generator. By \cite[Theorem~A.8]{SchwedeTucker2012} we have \(\Hom_A(A^{1/q},\omega_A)\cong\omega_{A^{1/q}}\). Using this with \(\omega_{A}\cong A(-L)\) and \(\omega_{A^{1/q}}\cong A^{1/q}(-L/q)\) we have
\begin{align*}
\Hom_{A}(A^{1/q},A)&\cong \Hom_{A}(A^{1/q},A(-L))(L),\\
&\cong\Hom_{A}(A^{1/q},\omega_{A})(L),\\
&\cong\omega_{A^{1/q}}(L),\\
&\cong A^{1/q}\left(\frac{q-1}{q}L\right).
\end{align*}
So \(\Hom_{A}(A^{1/q},A)\) is a free \(A^{1/q}\)-module, and letting its homogeneous generator be \(\Phi_q: A^{1/q}\to A\), every \(A\)-linear map \(\psi: A^{1/q}\to A\) is of the form \(\psi(x)=\Phi_q(a^{1/q}x)\) for some \(a\in A\). We have computed that \(\deg\Phi_q=-\frac{q-1}{q}L\).

Since by construction, \(A=A_{\mu'}\) is strongly \(F\)-regular and \(A\) is a domain, for any element \(a\neq0\) there exists some \(e\in\ZZ_{\ge1}\), \(q=p^e\), and an \(A\)-linear map \(\psi_a:A^{1/q}\to A\) such that \(\psi_a(a^{1/q})=1\). We apply strong \(F\)-regularity to \(g\coloneq f_1\cdots f_m\). So there exists \(\psi_g\in\Hom_A(A^{1/q},A)\) such that \(\psi_g(g^{1/q})=1\), and moreover we have there exists \(a\in A\) such that \(\psi_g(x)=\Phi_q(a^{1/q}x)\) for all \(x\in A^{1/q}\), so \(\psi_g(g^{1/q})=\Phi_q((ag)^{1/q})=1\).

Now write this \(a\in A\) in terms of the generators \(f_1,\dots,f_m\) as \(a=\sum_{v}\lambda_vf_1^{v_1}\cdots f_m^{v_m}\) with \(v=(v_1,\dots,v_m)\in\ZZ^{m}_{\ge0}\), \(\lambda_v\in k\). So since \(\mathrm{char}\,k=p>0\) we have
\begin{align*}
1=\Phi_q((ag)^{1/q})&=\Phi_q\left(\left(\sum_{v}\lambda_vf_1^{v_1+1}\cdots f_m^{v_m+1}\right)^{1/q}\right),\\
&=\Phi_q\left(\sum_v\left(\lambda_v f_1^{v_1+1}\cdots f_m^{v_m+1}\right)^{1/q}\right)=\sum_v\lambda_v^{1/q}\Phi_q\left((f_1^{v_1+1}\cdots f_m^{v_m+1})^{1/q}\right).
\end{align*}
Each summand above is \(\frac{1}{q}K\)-homogeneous, so by comparing degrees, there is at least one summand of degree \(0\), and since \(A_0=k\) we may assume it equals \(c\neq0\in k^\times\). Let such a summand be \(\lambda_u^{1/q}\Phi_q((f_1^{u_1}\cdots f_m^{u_m})^{1/q})=c\), it satisfies \(u_i\ge1\) for \(i=1,\dots,m\).

Computing degrees in the equation above, we have \(\sum_{i=1}^mu_i\frac{w_i}{q}-\frac{q-1}{q}L=0\), so that \(\sum_{i=1}^m \frac{u_i}{q-1}w_i=L\). Define \(0< a_i\coloneq \frac{u_i}{q-1}\) for \(i=1,\dots,m\). We have \(a_i\le 1\) because if \(a_i>1\) then \(u_i\ge q\) so by \(A\)-linearity \(k^\times\ni\Phi_q((f_1^{u_1}\cdots f_m^{u_m})^{1/q})=f_i\Phi_q((f_1^{u_1}\cdots f_i^{u_i-q}\cdots f_m^{u_m})^{1/q})\in\mathfrak{m}\), a contradiction. So we have obtained a decomposition of the anticanonical class
\[
\sum_{i=1}^ma_iw_i=[-K_X],\quad 0<a_i\le1\text{ for }i=1,\dots,m.
\]

Consider the local ring \(B\coloneq A_\mathfrak{m}\) and let \(\mathfrak{n}=\mathfrak{m}B\) be its maximal ideal. We have \(\dim B=\mathrm{ht}(\mathfrak{n})=\dim A= n+\rho(X)\), and \((B,\mathfrak{n})\) is \(F\)-finite, \(F\)-pure, and Cohen-Macaulay as localisation preserves these properties \cite[p.~294,Lemma~1.8,Remark~1.12]{SmithZhang2015}. Let \(\sigma\coloneq \sum_i u_i\), so that \(f_1^{u_1}\cdots f_m^{u_m}\in\mathfrak{n}^\sigma\). Now localising the map \(\Phi_q\) we obtain a \(B\)-linear map \(B^{1/q}\to B\) with \((f_1^{u_1}\cdots f_m^{u_m})^{1/q}\mapsto 1\). So for \(s\coloneq \sigma/(q-1)\), we have \((B,\mathfrak{n}^s)\) is sharply \(F\)-pure. By \Cref{fpt}, we have
\[
\sum_{i=1}^{m}a_i=s\le\mathrm{fpt}_B(\mathfrak{n})\le n+\rho(X).
\]
Furthermore, if \(R\) is not a polynomial ring then \(\dim R=\dim B<m\), and \(\mathfrak{n}/\mathfrak{n}^2\cong\mathfrak{m}/\mathfrak{m}^2\cong k^m\), so \(B\) is not regular, and by \Cref{fpt} we have \(\sum_{i=1}^ma_i\le\mathrm{fpt}_B(\mathfrak{n})\le n+\rho(X)-1\).
\end{proof}

\begin{proof}[Proof of \Cref{toricness criterion}] This is a corollary of \Cref{low coefficient sum}. We have \(\gamma_{\mathrm{Cox}}(X)\in\QQ\) because this quantity a linear program with rational objective functional and its region of optimisation is a polytope with rational vertices.
\end{proof}

We now obtain an upper bound on the Fano index. To do this, we use Gale duality to construct a linear functional \(\ell\in N_1(X)\) which may or may not coincide with an intersection pairing of the form \((\cdot,\gamma)\) for a rational curve \(\gamma\subset X\).  This emulates the way the generalised Mukai conjecture aims to bound the Fano index using the \emph{pseudo-index} \(\iota(X)\coloneq\min\{(-K_X\cdot\gamma)\mid\gamma\text{ rational curve in }X \}\). We prove \(i(X)\le\ell([-K_X])\) and express this in terms of the moment polytope. The upper bound proved is similar to the toric and spherical bounds \cite{Casagrande2006,Pasquier2010,GagliardiHofscheier2017,Fujita2019GeneralizedMukaiToric} proved using rational curves.

\begin{proposition}\label{Fano index} Let \(X\) be an ambiently smooth Fano variety. For any anticanonical representative of the form \(-K_X\sim_\QQ \mathcal{L}=\sum_{i=1}^ma_iD_i>0\) with \(0\le a_i\le1\) for all \(i=1,\dots,m\) (which we may write in the equivalent notation \(\mathcal{L}=\sum_{v\in\rho(\Sigma)}a_vD_v\)), we have for all \(\sigma\in\Sigma^{(d)}\) and all \(v\in\rho(\Sigma)\setminus\rho(\mathfrak{\sigma})\),
\[
i(X)\le \langle v,C_\sigma\rangle+a_v,
\]
where \(C_\sigma\) is the \(\QQ\)-Cartier data of \(\mathcal{L}\) on \(\sigma\).
\end{proposition}

\begin{proof}
    We begin by assuming that \(a_i>0\) for \(i=1,\dots,m\). Since \(Z\) is smooth, we have \(\Pic(X)=\Cl(X)=\Cl(Z)=\Pic(Z)=\langle\tau\rangle_\ZZ\) for all \(\tau\in\Phi(\mathcal{L})\). Therefore \([-K_X]\) may be viewed as a Cartier class on \(Z\), and \(\mathcal{L}=(i^*)^{-1}(-K_X)\) as a torus invariant \(\QQ\)-Cartier \(\QQ\)-divisor on \(Z\).  

    We now apply a weighted homogenisation operation to the rays of \(\Sigma\). It is defined by sending \(v\mapsto \overline{v}\coloneq(v,a_v)\) for all \(v\in\rho(\Sigma)\), and incorporating a new vector \(((0,\dots,0),1)\). That is, we consider the vector configuration
    \[
    \V\coloneq\{\overline{v}\mid v\in\rho(\Sigma)\}\cup\{((0,\dots,0),1)\}\subset N\times\QQ.
    \]
    Notice that the following vector configuration is Gale dual to \(\V\).
    \[
    \V^\downarrow=\{v^\downarrow\mid v\in\rho(\Sigma)\}\cup\{-\sum_{v\in\rho(\Sigma)}a_v v^\downarrow=[K_X]\}\subset K,
    \]
    with \(\overline{v}^\downarrow=v^\downarrow\) for \(v\in\rho(\Sigma)\), and \(((0,\dots,0),1)^\downarrow=[K_X]\).

    Since \(\mathcal{L}\) is a torus invariant \(\QQ\)-Cartier \(\QQ\)-divisor, it has a \(\QQ\)-Cartier data, which we write as \(\{C_\sigma\mid\sigma\in\Sigma^{(d)}\}\). These vectors satisfy \(\langle v,C_\sigma\rangle+a_v=0\) for each \(v\in\rho(\sigma)\), and since \([-K_X]\) is ample on \(Z\), we have further that this \(\QQ\)-Cartier data equals the vertices of the moment polytope \(\mathcal{P}\coloneq\mathcal{P}(Z,\mathcal{L})\). Next, we look at the polar dual of \(\mathcal{P}\)
    \[
    \mathcal{P}^*\coloneq\bigcap_{\sigma\in\Sigma^{(d)}}\{x\in N_\QQ\mid\langle x,C_\sigma\rangle+1\ge0\},
    \]
    which has set of vertices \(\{v/a_v\mid v\in\rho(\Sigma)\}\), and which is well-defined as \(a_v>0\) for all \(v\in\rho(\Sigma)\). We now consider the cone over this polytope
    \[
    \mathcal{C}\coloneq\cone\{(v/a_v,1)\mid v\in\rho(\Sigma)\}=\cone\{\overline{v}\mid v\in\rho(\Sigma)\}\subseteq N_\QQ\oplus\QQ.
    \]
    We have that
    \[
    \mathcal{C}=\bigcap_{\sigma\in\Sigma^{(d)}}\{(x,y)\in N_\QQ\oplus\QQ\mid\langle (x,y),(C_\sigma,1)\rangle\ge0\}
    \]
    has its facets cut out by the inner normal vectors \(\{(\mathcal{C}_\sigma,1)\mid\sigma\in\Sigma^{(d)}\}\).
    Therefore for each \(\sigma\in\Sigma^{(d)}\), the linear form \((C_\sigma,1)\in M_\QQ\oplus\QQ\) is non-negative on \(\mathcal{C}\), and on \(\V\) it attains the values \((C_\sigma,1)((0,\dots,0),1)=1\), and \((C_\sigma,1)(\overline{v})=\langle v,C_\sigma\rangle+a_v>0\) for \(v\in\rho(\Sigma)\setminus\rho(\sigma)\), and \((C_\sigma,1)(\overline{w})=0\) for each \(w\in\rho(\sigma)\). Therefore by Gale duality, this linear form induces the following linear relation in \(\V^\downarrow\),
    \[
    [K_X]+\sum_{v\in\rho(\Sigma)\setminus\rho(\sigma)}(\langle v,C_\sigma\rangle+a_v)v^\downarrow=0.\tag{\(\ast\)}\label{eqn}
    \]
    We now construct a linear form which extracts coefficients from equation \eqref{eqn}. Fix any \(v\in\rho(\Sigma)\setminus\rho(\sigma)\), then since \(\sigma\) is full dimensional,  we can find a linear relation
    \[
    v+\sum_{w\in\rho(\sigma)}\lambda_ww=0,\quad\lambda_w\in\QQ.
    \]
    By Gale duality on the rays of \(\Sigma\), the relation above induces a linear form \(\ell_{\sigma,v}\in L_\QQ\) such that \(\ell_{\sigma,v}(v^\downarrow)=1\), and \(\ell_{\sigma,v}(w^\downarrow)=\lambda_w\) for \(w\in\rho(\sigma)\), and \(\ell_{\sigma,v}(u^\downarrow)=0\) for all \(u\notin\rho(\sigma)\cup\{v\}\). 
    Therefore \(\ell_{\sigma,v}(\sigma^\downarrow)\subseteq\{0,1\}\subset\ZZ\), so \(\ell_{\sigma,v}\in\Hom(\langle\sigma^\downarrow\rangle_\ZZ,\ZZ)=\Cl(Z)^*=\Pic(Z)^*=L\). This means \(\ell_{\sigma,v}([-K_X])\in i(X)\cdot\ZZ\). Now we apply \(\ell_{\sigma,v}\) to \eqref{eqn} to obtain
    \[
    0< \langle C_\sigma,v\rangle+a_v=\ell_{\sigma,v}([-K_X])\in i(X)\cdot\ZZ.
    \]
    Therefore \(i(X)\le \langle C_\sigma,v\rangle+a_v\). We now consider the case where some coefficients \(a_i\) are zero. Let \(E\) be the divisor in \Cref{low coefficient sum}. Then for \(0<\varepsilon<1\) let \(-K_X\sim_\QQ \mathcal{L}_\varepsilon\coloneq (1-\varepsilon)\mathcal{L}+\varepsilon E>0\). Since \(\mathcal{L}_\varepsilon\sim_\QQ\mathcal{L}\), there is some \(y\in M_\QQ\) such that \(\mathcal{L}_\varepsilon=\mathcal{L}+\divv(f_y)=\sum_{v\in\rho(\Sigma)}(a_v+\langle v,y\rangle)D_v\eqcolon\sum_{v\in\rho(\Sigma)}a_v'D_v\), and the Cartier data of \(\mathcal{L}_\varepsilon\) satisfy \(C_\sigma^\varepsilon=C_\sigma-y\) for all \(\sigma\in\Sigma\). So we have \(\langle C_\sigma,v\rangle+a_v=\langle C_\sigma^\varepsilon,v\rangle+a_v'\) is constant for \(0<\varepsilon<1\), and all \(v\in \rho(\Sigma)\). Now take \(\varepsilon\to0\) so the result follows from the case with strictly positive coefficients.
\end{proof}

Next, using convex geometry of the moment polytope of \((Z,(i^*)^{-1}[-K_X])\), we connect the coefficient sum from the anticanonical divisor decomposition of \Cref{low coefficient sum} with the quantity \(i(X)\rho(X)\) using \Cref{Fano index} to obtain \(i(X)\rho(X)\le\sum_{i=1}^ma_i\le n+\rho(X)\). This approach is inspired by ideas of \cite{GagliardiHofscheier2017,Fujita2019GeneralizedMukaiToric} which originate from \cite{Casagrande2006}.

\begin{proposition}\label{inequality} Let \(X\) be an ambiently smooth Fano variety, then we have
    \[
    \mu(X)\ge\gamma_{\mathrm{Cox}}(X)\ge0.
    \]
\end{proposition}

\begin{proof} Fix \(\mathcal{L}=\sum_{i=1}^m a_iT_i\sim_\QQ (i^*)^{-1}(-K_X)\) where the \(a_i\) are as in \Cref{low coefficient sum}. We have \((a_1,\dots,a_m)\in Q^{-1}[-K_X]\cap[0,1]^m\), so by \Cref{low coefficient sum},
\[
    \gamma_\mathrm{Cox}(X)\ge \dim(X)+\rho(X)-\sum_{i=1}^{m}a_i\ge0.
\]
Now we write
\[
\gamma_{\mathrm{Cox}}(X)=\max\left\{\mathrm{dim}(X)+\rho(X)-\sum_{i=1}^ma_i\Biggm|(a_1,\dots,a_m)\in Q^{-1}[-K_X]\cap[0,1]^m \right\}.
\]
Suppose that \(Q(u)=Q(\mathcal{L}^*)\) for some \(u\in(\QQ^m)^*\), then since \(P^*:M_\QQ\to(\QQ^m)^*\) is injective, we have \(u=P^*y+\mathcal{L}^*\) for a unique \(y\in M_\QQ\). So \(\langle e_i,u\rangle=\langle e_i,P^*(y)\rangle+\langle e_i,\mathcal{L}^*\rangle=\langle P(e_i),y\rangle+a_i\). Therefore for all \(u\in Q^{-1}[-K_X]\cap[0,1]^m\) we have \(u=\sum_{i=1}^m\langle P(e_i),y\rangle e_i^*+\mathcal{L}^*\) for some \(y\in M_\QQ\). Recall that on \(Z\) the log pair \((Z,\sum_i^ma_iT_i)\) is log canonical if and only if \(0\le a_i\le 1\). So define \(\mathrm{LC}(Z,\mathcal{L})\coloneq\{y\in M_\QQ \mid 0\le \langle P(e_i),y\rangle+a_i\le 1\}\subseteq\mathcal{P}(Z,\mathcal{L})\). Therefore

\[
    \gamma_{\mathrm{Cox}}(X)=\max\left\{\mathrm{dim}(X)+\rho(X)-\sum_{i=1}^m a_i-\left\langle\sum_{i=1}^m P(e_i),y\right\rangle\Biggm|y\in\mathrm{LC}(Z,\mathcal{L})\right\}.
\]
Since \(\mathrm{LC}(Z,\mathcal{L})\subseteq\mathcal{P}(Z,\mathcal{L})=\conv\{C_\sigma\mid\sigma\in\Sigma^{(d)}\}\), for \(x\in\mathrm{LC}(Z,\mathcal{L})\) attaining the maximum in \(\gamma_{\mathrm{Cox}}(X)\) there is an expression
\[
x=\sum_{\sigma\in \Sigma^{(d)}}m_\sigma C_\sigma,
\]
with \(\sum_{\sigma\in\Sigma^{(d)}}m_\sigma=1\), and \(0\le m_\sigma\le1\) for all \(\sigma\in\Sigma^{(d)}\). We use that \(Z\) is smooth, so that \(|\rho(\Sigma) \setminus\rho(\sigma)|=\rho(X)\). We have
\begin{align*}
    \gamma_\mathrm{Cox}(X)&=\mathrm{dim}(X)+\rho(X)-\sum_{i=1}^ma_i-\left\langle \sum_{i=1}^m P(e_i),\sum_{\sigma\in\Sigma^{(d)}}m_\sigma C_\sigma\right\rangle,\\
    &=\mathrm{dim}(X)+\rho(X)-\sum_{i=1}^m\left(\sum_{\sigma\in\Sigma^{(d)}}m_\sigma a_i+\sum_{\sigma\in\Sigma^{(d)}}m_\sigma\langle P(e_i),C_\sigma\rangle\right),\\
    &=\mathrm{dim}(X)+\rho(X)-\sum_{i=1}^m\left(\sum_{\sigma\in\Sigma^{(d)};\, P(e_i)\notin\rho(\sigma)}m_\sigma(a_i+\langle P(e_i),C_\sigma\rangle)\right),\\
    &\le\mathrm{dim}(X)+\rho(X)-\sum_{i=1}^m\left(\sum_{\sigma\in\Sigma^{(d)};\, P(e_i)\notin\rho(\sigma)}m_\sigma i(X)\right),\\
    &=\mathrm{dim}(X)+\rho(X)-i(X)\cdot\sum_{\sigma\in\Sigma^{(d)}} m_\sigma|\rho(\Sigma) \setminus\rho(\sigma)|,\\
    &=\mathrm{dim}(X)+\rho(X)-i(X)\rho(X),\\
    &=\mu(X).\qedhere
\end{align*}
\end{proof}

Now we prove the boundary case of \Cref{main thm} characterising powers of projective spaces.

\begin{proposition}\label{equality} Let \(X\) be an ambiently smooth Fano variety such that \(\mu(X)=0\), then 
\[
X\cong(\PP^{i(X)-1})^{\rho(X)}.
\]
    
\end{proposition}

\begin{proof}[Proof (A)] In the proof of \Cref{inequality}, we found \(i(X)\rho(X)\le\sum_{i=1}^ma_i\), and in \Cref{low coefficient sum} we constructed the \(a_i\) such that 
\[
n+\rho(X)=i(X)\rho(X)\le\sum_{i=1}^ma_i\le\mathrm{fpt}(A_\mathfrak{m},\mathfrak{m}A_{\mathfrak{m}})\le\dim R=n+\rho(X)
\]
if and only if \(R\) is a polynomial ring, or equivalently, that \(X\) is a toric variety. Now the claim follows from the toric proof of the Mukai conjecture \cite{Casagrande2006}.
\end{proof}

\begin{proof}[Proof (B)] Assume \(X\) is smooth. For equality \(\mathrm{dim}(X)+\rho(X)-i(X)\rho(X)=0\) to hold, we must have equality in \Cref{Fano index}. That is, \(i(X)=\langle v,C_\sigma\rangle+a_v\) for all \(\sigma\in\Sigma^{(d)}\) and all \(v\in\rho(\Sigma)\setminus\rho(\sigma)\). Now fix \(\sigma\in\Sigma^{(d)}\), so we have
\[
\langle v,C_\sigma\rangle=\begin{cases}-a_v&\text{if }v\in\rho(\sigma),\\ i(X)-a_v&\text{otherwise.}\end{cases}
\]
Let \(\{w_1,\dots,w_{d}\}=\rho(\sigma)\), which by smoothness of \(Z\), form a \(\ZZ\)-basis of \(\ZZ^{d}\), so we have \(C_\sigma=-(a_{w_1}w_1^*+\dots+a_{w_{d}}w_{d}^*)\). Next, for any \(\sigma'\in\Sigma^{(d)}\) we find
\begin{align*}
C_{\sigma'}&=\sum_{w\in\rho(\sigma)}\langle w,C_{\sigma'}\rangle w^*,\\
&=-\sum_{w\in\rho(\sigma)\cap\rho(\sigma')}a_{w}w^*+\sum_{w\in\rho(\sigma)\setminus\rho(\sigma')}(i(X)-a_w)w^*,\\
&=C_\sigma+i(X)\cdot\sum_{w\in\rho(\sigma)\setminus\rho(\sigma')}w^*,
\end{align*}
and we now use this formula to determine the locations of all \(v\in\rho(\Sigma)\). Since \(\Sigma\) is a complete simplicial fan, for all \(i=1,\dots,d\) we can define a unique \(\sigma_i\in\Sigma^{(d)}\) such that \(\sigma\cap\sigma_i\) is a facet of \(\sigma\) and \(w_i\notin\rho(\sigma_i)\). Then for any \(v\in\{v_1,\dots,v_{\rho(X)}\}\coloneq\rho(\Sigma)\setminus\{w_1,\dots,w_{d}\}\) we have
\[
\langle v,w_i^*\rangle=\langle v,C_{\sigma_i}-C_{\sigma}\rangle/{i(X)}=(\langle v,C_{\sigma_i}\rangle-(i(X)-a_v))/i(X)=\begin{cases} -1&\text{if } v\in \rho(\sigma_i),\\0&\text{otherwise.}\end{cases}
\]
Therefore, in the coordinate system \(\{w_1,\dots,w_d\}\) we can write \(v=-\sum_{i\in I(v)}w_i\) where \(I(v)\coloneq\{i=1,\dots,d\mid v\in\rho(\sigma_i)\}\). Since \(\Sigma\) is simplicial, \(v\) is uniquely determined by the property that \(v\in\rho(\sigma_i)\) for any one of the \(i=1,\dots,d\), so we have \(I(v)\cap I(v')=\emptyset\) for \(v\neq v'\), so \(\{1,\dots,d\}=\bigsqcup_{\ell=1}^{\rho(X)}I(v_\ell)\) is a partition. We have found that
\[
\rho(\Sigma)=\bigsqcup_{\ell=1}^{\rho(X)}\left(\{w_i\mid i\in I(v_\ell)\}\cup\{-\sum_{i\in I(v_\ell)}w_i\}\right)=\rho(\Sigma(\PP^{|I(v_1)|}\times\dots\times\PP^{|I(v_{\rho(X)})|})).
\]
That is, for \(\mathrm{Cox}(X)=\CC[T_1,\dots,T_m]/\langle g_1,\dots,g_r\rangle\), we have \(\CC[T_1,\dots, T_m]\) is the graded Cox ring of a product of projective spaces. So \(\{\deg(T_i)\mid i=1,\dots,m\}=\{u_1,\dots,u_{\rho(X)}\}\) is a \(\ZZ\)-basis of \(\ZZ^{\rho(X)}\). Therefore the ambient toric variety is a product of projective spaces \(Z\cong\PP^{n_1}\times\dots\times \PP^{n_{\rho(X)}}\).

Therefore the effective cone equals the nef cone \(\Eff(X)=\Nef(X)\), so that all the extremal rays \(R_1,\dots, R_{\rho(X)}\) are of fibre type. Then \(\ell(R_j)\ge i(X)\) for \(j=1,\dots,\rho(X)\). So we have \(n=(i(X)-1)\rho(X)\le\sum_{j=1}^{\rho(X)} (\ell(R_j)-1)\). Therefore since \(X\) is smooth, by \cite[Theorem~2.11]{Fujita2014} we have \(X\cong(\PP^{i(X)-1})^{\rho(X)}\).
\end{proof}
This completes the proof of the Mukai conjecture for the ambiently smooth Fano varieties.

\begin{proof}[Proof of \Cref{main thm}] This is immediate from \Cref{inequality} and \Cref{equality}. 
\end{proof}

\section{A remark on the sum of lengths of extremal rays minus one}\label{s section}

Consider the following invariant introduced by K. Fujita \cite{Fujita2014} for the smooth Fano varieties. This invariant is related to the Mukai conjecture, for example see the last paragraph of Proof (B) of \Cref{equality}. Let \(n\coloneq\mathrm{dim}(X)\). Let \(X\) be a smooth Fano variety. Recall from \Cref{s definition},
\[
s(X)\coloneq\sum_{R\text{ extremal ray}}(\ell(R)-1).
\]

Motivated by the Mukai conjecture, it is natural to look for a meaningful upper bound on \(s(X)\), and to classify the varieties which achieve extremal values of \(s(X)\). It is known that \(s(X)\le n\) when \(n\le 4\) \cite{Fujita2014}, and by checking the complete database of low-dimensional smooth toric Fano varieties we find \(s(X)\le n\) when \(n\le8\) (available from their Magma databases \cite{MagmaSmoothFano}).

In general there are examples with \(s(X)>n\), for example \cite[Remark~3.5~iii]{Fujita2014} gives examples \(X_n\) with \(s(X_n)=2n-4>n\) for \(n\ge 5\). This example is non-toric, and in dimensions \(\ge 9\) we also find toric examples with \(s(X)>n\).

Here we show that among the toric varieties, there is no upper bound on \(s(X)\) which is polynomial in the dimension.

\begin{lemma} Let \(X\) be a toric Fano variety. There exists a toric Fano variety \(X'\) with \(\Cl(X')\cong \Cl(X)\), \(\mathrm{dim}(X')=2\mathrm{dim}(X)+\rk\Cl(X)\). We have \(\Nef(X)=\Nef(X')\), so there is a bijection between the extremal rays \(R\) of \(X\) and the extremal rays \(R'\) of \(X'\), and we have \(\ell(R')=2\ell(R)\) for each extremal ray \(R\). Moreover, if \(X\) is smooth (resp. \(\QQ\)-factorial), then \(X'\) is smooth (resp. \(\QQ\)-factorial).
    
\end{lemma}

\begin{proof} Let \(X\) have \(\Cl(X)\)-graded Cox ring \(\mathrm{Cox}(X)=\CC[T_1,\dots,T_m]\). We define \(X'\) by its Cox ring. Let \(R=\CC[T_1,S_1,\dots,T_m,S_m]\), graded by \(\deg(T_i)=\deg(S_i)\) for \(i=1,\dots,m\), where \(\deg\) is the \(\Cl(X)\)-grading of \(\mathrm{Cox}(X)\), so \(\Cl(X')\cong\Cl(X)\). We have \(\mathrm{dim}(X')=2m-\rk\Cl(X')=2\mathrm{dim}(X)+\rk\Cl(X)\) and \([-K_{X'}]=\sum_{i=1}^m \deg(T_i)+\sum_{i=1}^m\deg(S_i)=2[-K_X]\). Since \(\QQ_{\ge0}[-K_X]=\QQ_{\ge0}[-K_X']\) is chosen to be spanned by an ample class by the construction of \Cref{background section}, \(X'\) is Fano, and the bunched cones in the bunched polynomial ring of \(X'\) coincide with the bunched cones of \(X\), so smoothness or \(\QQ\)-factoriality of \(X\) is preserved in \(X'\), we have \(\Nef(X')=\Nef(X)\) and \(\Pic(X')=\Pic(X)\). So it is clear that \(\ell(R')=2\ell(R)\).
\end{proof}

\begin{proposition} For the class of smooth toric Fano varieties there is no upper bound on \(s(\cdot)\) which is polynomial in \(\dim(\cdot)\).
    
\end{proposition}

\begin{proof} We exhibit a sequence of smooth toric Fano varieties \(\{X_k\mid k\ge1\}\) such that \(\{s(X_k)\mid k\ge1\}\) has exponential growth in \(\dim(X_k)\). For \(k\ge 1\), let \(V^{2k}\) be the Voskresenskij-Klyachko \(2k\)-fold. It is the smooth toric Fano variety whose fan has rays \(\{e_1,\dots,e_{2k},-e_1,\dots,-e_{2k},e_1+\dots+e_{2k},-(e_1+\dots+e_{2k})\}\). As noted by \cite{Casagrande2003Symmetric}, \(V^{2k}\) has a large number of extremal rays, all of which have length one. It is immediate from the description in \cite{Casagrande2003Symmetric} that \(\#\{\text{extremal rays of \(V^{2k}\)}\}=\binom{2k+2}{k+1}\), \(\dim V^{2k}=2k\), and \(\rk\Cl(V^{2k})=2k+2\). We define \(X_k=(V^{2k})'\) for \(k\ge1\). Then \(\mathrm{dim}(X_k)=6k+2\) and \(s(X_k)=\#\{\text{extremal rays of \(V^{2k}\)}\}=\binom{2k+2}{k+1}\). So \(\lim_{k\to\infty}\frac{s(X_k)}{\dim(X_k)^\alpha}=\infty\) for all \(\alpha\ge1\).
\end{proof}

\printbibliography{}

@book{CoxRingBook,
  title     = {Cox Rings},
  author    = {Arzhantsev, Ivan and Derenthal, Ulrich and Hausen, J\"urgen and Laface, Antonio},
  series    = {Cambridge Studies in Advanced Mathematics},
  volume    = {144},
  publisher = {Cambridge University Press},
  year      = {2014},
  doi       = {10.1017/CBO9781139175852},
  isbn      = {9781107024625}
}

@misc{NefComplexity,
  title        = {Characterization of products of projective spaces via nef complexity},
  author       = {Enwright, J. and Filipazzi, S. and Gongyo, Y. and Moraga, J. and Svaldi, R. and Wang, C. and Watanabe, K.},
  year         = {2025},
  eprint       = {2512.13637},
  archivePrefix= {arXiv},
  primaryClass = {math.AG},
  doi          = {10.48550/arXiv.2512.13637}
}

@misc{SphericalGeneralisedMukai,
  title        = {The generalised Mukai conjecture for spherical varieties},
  author       = {Gagliardi, Giuliano and Hofscheier, Johannes and Pearson, Heath},
  year         = {2025},
  eprint       = {2502.21155},
  archivePrefix= {arXiv},
  primaryClass = {math.AG},
  doi          = {10.48550/arXiv.2502.21155}
}

@article{Andreatta2004GeneralizedMukai,
  title   = {Generalized Mukai conjecture for special Fano varieties},
  author  = {Andreatta, Marco and Chierici, Elena and Occhetta, Gianluca},
  journal = {Central European Journal of Mathematics},
  volume  = {2},
  number  = {2},
  pages   = {272--293},
  year    = {2004}
}

@article{Bonavero2003ConjectureMukai,
  title   = {Sur une conjecture de Mukai},
  author  = {Bonavero, Laurent and Casagrande, Cinzia and Debarre, Olivier and Druel, St\'ephane},
  journal = {Commentarii Mathematici Helvetici},
  volume  = {78},
  number  = {3},
  pages   = {601--626},
  year    = {2003},
  doi     = {10.1007/s00014-003-0765-x}
}

@article{Fujita2019GeneralizedMukaiToric,
  title   = {The generalized Mukai conjecture for toric log Fano pairs},
  author  = {Fujita, Kento},
  journal = {European Journal of Mathematics},
  volume  = {5},
  number  = {3},
  pages   = {858--871},
  year    = {2019},
  doi     = {10.1007/s40879-018-0302-5}
}

@article{Fujita2016AroundMukai,
  title   = {Around the Mukai conjecture for Fano manifolds},
  author  = {Fujita, Kento},
  journal = {European Journal of Mathematics},
  volume  = {2},
  number  = {1},
  pages   = {120--139},
  year    = {2016},
  doi     = {10.1007/s40879-015-0045-5}
}

@article{Wisniewski1990Mukai,
  title   = {On a conjecture of Mukai},
  author  = {Wi{\'s}niewski, Jaros{\l}aw A.},
  journal = {Manuscripta Mathematica},
  volume  = {68},
  number  = {2},
  pages   = {135--141},
  year    = {1990},
  doi     = {10.1007/BF02568756}
}

@article{Reineke2024MukaiQuiver,
  title   = {The Mukai conjecture for Fano quiver moduli},
  author  = {Reineke, Markus},
  journal = {Algebra and Representation Theory},
  volume  = {27},
  number  = {4},
  pages   = {1641--1644},
  year    = {2024},
  doi     = {10.1007/s10468-024-10268-8}
}

@incollection{MukaiConjecture,
  title     = {Problems on characterization of the complex projective space},
  author    = {Mukai, Shigeru},
  booktitle = {Birational Geometry of Algebraic Varieties, Open Problems},
  series    = {Proceedings of the 23rd Symposium of the Taniguchi Foundation},
  year      = {1988},
  pages     = {57--60},
  address   = {Katata, Japan}
}

@article{Geomcharoftoric,
	author = {Brown, Morgan and McKernan, James and Svaldi, Roberto and Zong,Hong},
	doi = {10.1215/00127094-2017-0047},
	journal = {Duke Mathematical Journal},
	number = {5},
	pages = {923-968},
	title = {A geometric characterization of toric varieties},
	volume = {167},
	year = {2018}}

@book{toricvarieties,
  title     = {Toric Varieties},
  author    = {Cox, David A. and Little, John B. and Schenck, Hal K.},
  series    = {Graduate Studies in Mathematics},
  volume    = {124},
  year      = {2011},
  publisher = {American Mathematical Society},
  address   = {Providence, RI}
}

@article{Casagrande2006,
  author  = {Casagrande, Cinzia },
  title   = {The number of vertices of a {Fano} polytope},
  journal = {Annales de l'Institut Fourier},
  volume  = {56},
  number  = {1},
  pages   = {121--130},
  year    = {2006},
  doi     = {10.5802/aif.2175},
}

@article{BCHM2010,
  author  = {Birkar, Caucher and Cascini, Paolo and Hacon, Christopher D. and McKernan, James},
  title   = {Existence of minimal models for varieties of log general type},
  journal = {Journal of the American Mathematical Society},
  volume  = {23},
  number  = {2},
  pages   = {405--468},
  year    = {2010},
  doi     = {10.1090/S0894-0347-09-00649-3}
}

@article{Fujita2014,
  author  = {Fujita, Kento},
  title   = {On a generalization of the {Mukai} conjecture
             for {Fano} fourfolds},
  journal = {Tokyo Journal of Mathematics},
  volume  = {37},
  number  = {2},
  year    = {2014},
  pages   = {319--333},
  doi     = {10.3836/tjm/1422452796}
}

@misc{MagmaSmoothFano,
  author       = {{\O}bro, Mikkel},
  title        = {Smooth {Fano} polytope databases},
  howpublished = {Magma database collection, University of Sydney},
  url          = {https://magma.maths.usyd.edu.au/magma/download/db/},
  note         = {Archives smoothfano.tar.gz, smoothfano7.tar.gz, smoothfano8.tar.gz, Accessed: 23rd September 2026}
}

@article{Casagrande2003Symmetric,
  author  = {Casagrande, Cinzia},
  title   = {Centrally symmetric generators in toric {Fano} varieties},
  journal = {Manuscripta Mathematica},
  volume  = {111},
  number  = {4},
  pages   = {471--485},
  year    = {2003},
  doi     = {10.1007/s00229-003-0374-4}
}

@article{GagliardiHofscheier2017,
  author  = {Gagliardi, Giuliano and Hofscheier, Johannes},
  title   = {The generalized {Mukai} conjecture for symmetric varieties},
  journal = {Transactions of the American Mathematical Society},
  volume  = {369},
  number  = {4},
  year    = {2017},
  pages   = {2615--2649},
  doi     = {10.1090/tran/6738}
}

@article{Pasquier2010,
  author  = {Pasquier, Boris},
  title   = {The pseudo-index of horospherical {Fano} varieties},
  journal = {International Journal of Mathematics},
  volume  = {21},
  number  = {9},
  year    = {2010},
  pages   = {1147--1156},
  doi     = {10.1142/S0129167X10006422}
}

@article{Gagliardi2014,
  author  = {Gagliardi, Giuliano},
  title   = {The {Cox} ring of a spherical embedding},
  journal = {Journal of Algebra},
  volume  = {397},
  year    = {2014},
  pages   = {548--569},
  doi     = {10.1016/j.jalgebra.2013.08.037}
}

@incollection{SmithZhang2015,
  author    = {Smith, Karen E. and Zhang, Wenliang},
  title     = {{Frobenius} splitting in commutative algebra},
  booktitle = {Commutative Algebra and Noncommutative Algebraic Geometry.
               Vol. I: Expository Articles},
  editor    = {Eisenbud, David and Iyengar, Srikanth B. and
               Singh, Anurag K. and Stafford, J. Toby and
               Van den Bergh, Michel},
  series    = {Mathematical Sciences Research Institute Publications},
  volume    = {67},
  publisher = {Cambridge University Press},
  address   = {Cambridge},
  year      = {2015},
  pages     = {291--345},
  doi       = {10.1017/9781107588530.011}
}

@article{SinghTakagiVarbaro2017,
  author  = {Singh, Anurag K. and Takagi, Shunsuke and Varbaro, Matteo},
  title   = {A {Gorenstein} criterion for strongly {$F$}-regular
             and log terminal singularities},
  journal = {International Mathematics Research Notices},
  volume  = {2017},
  number  = {21},
  pages   = {6484--6522},
  year    = {2017},
  doi     = {10.1093/imrn/rnw195}
}

@article{Brown2013,
  author  = {Brown, Morgan},
  title   = {Singularities of {Cox} rings of {Fano} varieties},
  journal = {Journal de Math{\'e}matiques Pures et Appliqu{\'e}es},
  volume  = {99},
  number  = {6},
  pages   = {655--667},
  year    = {2013},
  doi     = {10.1016/j.matpur.2012.10.003}
}

@article{Canonicalmodule,
  author  = {Hashimoto, Mitsuyasu and Kurano, Kazuhiko},
  title   = {The canonical module of a {Cox} ring},
  journal = {Kyoto Journal of Mathematics},
  volume  = {51},
  number  = {4},
  pages   = {855--874},
  year    = {2011},
  doi     = {10.1215/21562261-1424884}
}

@article{GotoWatanabe1978II,
  author  = {Goto, Shiro and Watanabe, Keiichi},
  title   = {On graded rings, {II}
             ({$\mathbb{Z}^n$}-graded rings)},
  journal = {Tokyo Journal of Mathematics},
  volume  = {1},
  number  = {2},
  year    = {1978},
  pages   = {237--261},
  doi     = {10.3836/tjm/1270216496}
}

@article{Takagi2013,
  author  = {Takagi, Shunsuke},
  title   = {Adjoint ideals and a correspondence between
             log canonicity and {$F$}-purity},
  journal = {Algebra \& Number Theory},
  volume  = {7},
  number  = {4},
  year    = {2013},
  pages   = {917--942},
  doi     = {10.2140/ant.2013.7.917}
}

@book{BrunsHerzog1998,
  author    = {Bruns, Winfried and Herzog, J{\"u}rgen},
  title     = {Cohen--Macaulay Rings},
  edition   = {Revised},
  series    = {Cambridge Studies in Advanced Mathematics},
  volume    = {39},
  publisher = {Cambridge University Press},
  year      = {1998},
  doi       = {10.1017/CBO9780511608681}
}

@article{FHN2018,
  author  = {Fahrner, Anne and Hausen, J{\"u}rgen and Nicolussi, Michele},
  title   = {Smooth projective varieties with a torus action
             of complexity 1 and {Picard} number 2},
  journal = {Ann. Sc. Norm. Super. Pisa Cl. Sci. (5)},
  volume  = {18},
  number  = {2},
  year    = {2018},
  pages   = {611--651},
  doi     = {10.2422/2036-2145.201608_024}
}

@misc{OAI:The-Generalized-Mukai-Conjecture-September-24-2026,
  author = {{OpenAI}},
  title = {{The generalized Mukai conjecture}},
  howpublished = {OpenAI Math Release preprint
                  \href{https://github.com/openai/math/blob/main/preprints/The-Generalized-Mukai-Conjecture-September-24-2026/article.pdf}{OAI:The-Generalized-Mukai-Conjecture-September-24-2026}},
  note = {Accessed: 7th October 2026}
}

@book{Eisenbud1995,
  author    = {Eisenbud, David},
  title     = {Commutative Algebra with a View Toward Algebraic Geometry},
  series    = {Graduate Texts in Mathematics},
  volume    = {150},
  publisher = {Springer-Verlag},
  address   = {New York},
  year      = {1995},
  doi       = {10.1007/978-1-4612-5350-1}
}

@incollection{SchwedeTucker2012,
  author    = {Schwede, Karl and Tucker, Kevin},
  title     = {A Survey of Test Ideals},
  booktitle = {Progress in Commutative Algebra 2:
               Closures, Finiteness and Factorization},
  publisher = {Walter de Gruyter},
  address   = {Berlin},
  year      = {2012},
  pages     = {39--99},
  doi       = {10.1515/9783110278606.39}
}

@article{KobayashiOchiai1973,
  author  = {Kobayashi, Shoshichi and Ochiai, Takushiro},
  title   = {Characterizations of complex projective spaces and hyperquadrics},
  journal = {J. Math. Kyoto Univ.},
  volume  = {13},
  number  = {1},
  year    = {1973},
  pages   = {31--47},
  doi     = {10.1215/kjm/1250523432}
}

@book{Fujita1990,
  author    = {Fujita, Takao},
  title     = {Classification Theories of Polarized Varieties},
  series    = {London Mathematical Society Lecture Note Series},
  volume    = {155},
  publisher = {Cambridge University Press},
  address   = {Cambridge},
  year      = {1990},
  doi       = {10.1017/CBO9780511662638}
}

@article{Mukai1989,
  author    = {Mukai, Shigeru},
  title     = {Biregular classification of {Fano} 3-folds
               and {Fano} manifolds of coindex 3},
  journal   = {Proc. Natl. Acad. Sci. USA},
  volume    = {86},
  number    = {9},
  year      = {1989},
  pages     = {3000--3002},
  doi       = {10.1073/pnas.86.9.3000}
}

@article{Mella1999,
  author    = {Mella, Massimiliano},
  title     = {Existence of good divisors on {Mukai} varieties},
  journal   = {J. Algebraic Geom.},
  volume    = {8},
  number    = {2},
  year      = {1999},
  pages     = {197--206}
}

@misc{GHP2026,
  author        = {Gagliardi, Giuliano and Hofscheier, Johannes
                   and Pearson, Heath},
  title         = {Toricness and smoothness criteria for spherical varieties},
  year          = {2026},
  eprint        = {2601.06376},
  archivePrefix = {arXiv},
  primaryClass  = {math.AG},
  doi           = {10.48550/arXiv.2601.06376}
}

@article{Gongyo2026,
  author  = {Gongyo, Yoshinori},
  title   = {Effective non-vanishing and the {Mukai}-type conjecture},
  journal = {Kyoto Journal of Mathematics},
  volume  = {66},
  number  = {3},
  year    = {2026},
  pages   = {643--653},
  doi     = {10.1215/21562261-2024-0059}
}

@misc{GMMS2023,
  author       = {Gongyo, Yoshinori and Mauri, Mirko
                  and Moraga, Joaqu{\'i}n and Svaldi, Roberto},
  title        = {Higher-dimensional log {Calabi--Yau} pairs},
  year         = {2023},
  howpublished = {American Institute of Mathematics workshop proposal},
  url          = {https://aimath.org/~farmer/tmp/workshops2023b.pdf}
}

@article{Braun2022,
  author  = {Braun, Lukas},
  title   = {Gorensteinness and iteration of {Cox} rings for {Fano} type varieties},
  journal = {Mathematische Zeitschrift},
  volume  = {301},
  pages   = {1047--1061},
  year    = {2022},
  doi     = {10.1007/s00209-021-02946-w}
}

@article{DHHKL2015,
  author  = {Derenthal, Ulrich and Hausen, J{\"u}rgen and
             Heim, Armand and Keicher, Simon and Laface, Antonio},
  title   = {{Cox} rings of cubic surfaces and {Fano} threefolds},
  journal = {Journal of Algebra},
  volume  = {436},
  year    = {2015},
  pages   = {228--276},
  doi     = {10.1016/j.jalgebra.2015.04.028}
}

@article{Suess2014,
  author  = {S{\"u}{\ss}, Hendrik},
  title   = {{Fano} threefolds with {$2$}-torus action:
             a picture book},
  journal = {Documenta Mathematica},
  volume  = {19},
  year    = {2014},
  pages   = {905--940},
  doi     = {10.4171/DM/468}
}

@article{Suess2018,
  author  = {S{\"u}{\ss}, Hendrik},
  title   = {Erratum for ``{Fano} threefolds with {$2$}-torus action''},
  journal = {Documenta Mathematica},
  volume  = {23},
  year    = {2018},
  pages   = {2079--2080},
  doi     = {10.4171/DM/669}
}

@incollection{BatyrevPopov2004,
  author    = {Batyrev, Victor V. and Popov, Oleg N.},
  title     = {The {Cox} ring of a del {Pezzo} surface},
  booktitle = {Arithmetic of Higher-Dimensional Algebraic Varieties},
  series    = {Progress in Mathematics},
  volume    = {226},
  publisher = {Birkh{\"a}user Boston},
  year      = {2004},
  pages     = {85--103},
  doi       = {10.1007/978-0-8176-8170-8_5}
}

\end{document}